\def\ch{{\rm ch}}
\def\sh{{\rm sh}}

\documentclass [11pt] {article}

\usepackage{amssymb,amsfonts,amsmath,amsthm}
\usepackage[cp866nav]{inputenc}
\usepackage[english]{babel}

\newtheorem{remark}{\large Remark}

\newtheorem{theorem}{\large Theorem}
\newtheorem{corollary}{\large Corollary}
\newtheorem{lemma}{\large Lemma}

\begin{document}

\title
{Busy period,  time of the first loss of a customer   and the number of
 customers in $ {\rm M^{\varkappa}|G^{\delta}|1|B} $ }

\date{ }
\author {V. Kadankov  \thanks{Institute of Mathematics
of the  Ukrainian National Academy of Sciences 3, Tereshchenkivska
st. 01601, Kyiv-4, Ukraine; phone:  452-00-55. E-mail:
kadankov@voliacable.com.}\;
 T. Kadankova
\thanks{Vrije Universiteit Brussel, Department of Mathematics, 
 Building G, Pleinlaan 2,  \newline 1050  Brussels,  Belgium,  tel.: +32(0) 2 6626   34 68  ,\: e-mail:
tetyana.kadankova@vub.ac.be}
 N. Veraverbeke
\thanks{Hasselt University,  Center for Statistics, Building D, 3590
Diepenbeek, Belgium,  \newline tel.: +32(0)11 26 82 37,\: e-mail:
noel.veraverbeke@uhasselt.be}}

\maketitle


\noindent{\bf Running head:}
Busy period,  time of the first loss of a customer and the number of
  customers
\\ {\bf 2000 Mathematics subject classification: } 60G40;
60K20
\begin{abstract}
A two-sided exit problem is solved for a difference of a compound Poisson
process and a compound renewal process. More precisely, the Laplace transforms
of the joint distribution of the first exit time, the value of the
overshoot  and the value of a linear component at this  instant
are found.
Further, we   study   the process reflected in its supremum.
We  determine the  main two-boundary  characteristics of the process
reflected in its supremum.
These results are then  applied for studying the
$ {\rm M^{\varkappa}|G^{\delta}|1|B}$ system.
We  derive the  distribution of a  busy period and the numbers of customers
in the system in  transient and stationary regimes.
The advantage is that these results are in a closed form, in terms of
resolvent sequences of the process.
\end{abstract}

{\bf\Large  Introduction}
Queueing systems with batch arrivals and finite buffer have wide applications
in the performance evaluation, telecommunications and manufacturing systems.
One of the crucial performance issues of the single-server queue with finite
buffer  room is losses, namely, customers (packets, cells, jobs) that are not
allowed to enter the system due to the buffer overflow. This issue is
especially important in the analysis of telecommunication networks. Motivated
by this fact, we derived the most important performance measures of  queueing
systems of this type. More precisely,  we considered the
$  M^{\varkappa}|G^{\delta}|1|B $  queueing system with finite buffer and its
modification.
We consider partial rejection, meaning  that if an
overflow of  buffer  occurs due to the arrival of a  batch of customers,
the amount of work brought by this batch  is only partially admitted to  the
buffer, up to the limit of the free buffer space just before the arrival.
The rest is rejected and  therefore, is lost.

Evolution of the number of  customers in such systems is described by a
process with two reflecting  boundaries. In general case this process is a
difference of two  compound renewal processes. Reflections from the upper
boundary are generated by the supremum (infimum) of the process. Reflections
from the lower boundary govern the server's behavior. In general, such processes
are not Markovians, but by  adding a complementary linear component
(in some literature  called  age process), we obtain a Markov process, which
describes  functioning of the queueing system.
Studying  main characteristics of the system results to the investigating  the
two-boundary functionals of the governing  process.
We  applied the solutions  of the  two-sided  exit  problem  for the governing
process (see  \cite{2Ka13} and \cite{2Ka2} for the  methodology)
to obtain performance measures of interest.
For the queueing system of $ M^{\varkappa}|G^{\delta}|1|B,$
$ {\rm G^{\delta}| M^{\varkappa}|1|B} $ (see\cite{2Ka17})
type the governing
process is the difference of the compound Poisson process and the compound
renewal process complemented with the age process.
The main result of this paper is the closed form formulae for the Laplace
transforms of the busy period, time of the first loss and the number of
customers in the system  at arbitrary time.

First passage times of the level by  L\'{e}vy processes in context of queues
were  considered  in  \cite{Dube2004}, where the  explicit characterization
of the Laplace transform of the busy period distribution was found for a  finite
capacity  $ M|G|1 $ queue, see also \cite{Perry2000}, 
\cite{Chydzinski2007}  and \cite{Bekker2008}.
In regard to finding the buffer overflow time we refer to  \cite{Asmussen2002}, where arrivals are modeled  by a Markov modulated Poisson  process  (MMPP) and  service  time is  exponential, and also  to  \cite{Chydzinski2007}, where $ BMAP|G|1|b $ queue was considered. Previous works on
the overflow period were concentrated on simple Poisson arrivals
\cite{Boer2001}, \cite{Chydzinski2004}, 
or renewal arrivals  \cite{Fakinos1982}
\cite{Pacheco2008}.
Picard and Lef\`{e}vre \cite{Lefevre 1} found the distribution of the first
crossing time of a Poisson process and renewal process  in terms of polynomials
of Abel-Gontcharoff types.

In recent years there has been a great interest in analyzing various queueing
models with MAP (Markov Arrival  Process) as input process or  MSP
(Markov service process). MAP is used to represent  correlated traffic arising
in modern telecommunication networks.
In  systems  with   Markov  arrival  or service  processes (MAP, BMAP, or BMSP)
and their modifications, it is  common  to  use the supplementary variable
methods and/or  embedded Markov chains.
For the method of supplementary variable we refer, for instance, to
\cite{Gupta2006}, where   $GI|MSP|1$ queue with finite as well as infinite
buffer was analyzed.
Embedded Markov  chains techniques were used  by \cite{Dukhovny1996}.
De  Boer  et.al.  \cite{Boer2001}
studied stationary distribution of the remaining service time upon reaching some
target level in an $ M|G|1$ system. The asymptotic analysis of $G|MSP|1|r$  queue
has been carried out by \cite{Bocharov2003}.  Banik et.al. \cite{Banik2008}
found steady state  distribution  for a finite-buffer single-server queue
$GI|BMSP|1|N$ with renewal input.  Random size batch service queueing models
were  subject  of study in  \cite{Economou2003},  \cite{Chaplygin2003}
(stationary analysis  of $GI|BMSP|1$  queue)  and \cite{Kim2007}
(asymptotic behavior of the loss probability).

Hence,  the majority of  recent  literature is devoted mainly to the queue size and
workload, most of the  times in the steady state  case. However, recently it
was shown that steady-state parameters  do not reflect the reality. A detailed
discussion of the drawbacks of steady-state parameters in telecommunication networks may be found in  \cite{Schwefel2001}.
This remark   emphasizes  importance  of studying  main  performance  measures
in transient  regime, which is  the  topic  of  this article.
The main two-boundary characteristic of the governing process is the joint
distribution of $ \{\chi,L,T\} $, i.e. of the first exit time from the interval,
the value of the overshoot and the value of the linear component at this
instant.
For the overview of  the existing results on the two-boundary
problems we refer to \cite{2Ka13}. And  here we only mention several
authors who contributed a lot in the development
of this  area.  Starting  from Kemperman (1963), Takacs  (1966),
Emery \cite{Emery}, Pecherskii \cite{Pe},  Suprun, Shurenkov
(\cite{Su}, \cite{Sh}), Lambert \cite{Lamb1}, Doney \cite{Don1},
Avram, Kyprianou, Pistorious \cite{Avram2004}, Pistorious \cite{Pist2},
Kyprianou, Palmowsky  \cite{Kyp1}, and Kadankov, Kadankova (\cite{2Ka2},
\cite{2Ka15}) studied one-  and  two-boundary characteristics for  different classes of stochastic processes.

The Laplace transforms of the joint distribution of the first exit time and
the value of the overshoot at this time instant for general  L\'{e}vy processes
and random walks have been determined in \cite{2Ka2}. The Laplace transforms
of this joint distribution were found in terms of the Laplace transforms of the
one-boundary characteristics of the process. This method for L\'{e}vy processes and
random walks  \cite{2Ka2} was then applied for other classes of stochastic
processes, such as the difference of compound renewal processes \cite{2Ka15},
 and semi-Markov random walks with linear drift  \cite{KAD44}.

The rest of the article is structured as follows.
In Section 1   we introduce  the process,  necessary notation and consider
the   one-boundary characteristics of the process.
The two-sided problem is solved in Section 2. In this section we also  prove
the  weak convergence  of the joint  distribution of the supremum, infimum
and the value of the process to the corresponding distribution of the symmetric Wiener process.
Section 3  deals with  the reflected processes. We consider the  first  passage of the lower  boundary,
distribution of the increments of the  process  and its asymptotic behavior.
Finally, in Section 4 we apply the results obtained in the previous sections
for studying  several characteristics of the queueing system $ {\rm M^{\varkappa}|G^{\delta}|1|B}, $
such as busy period,  time of the first loss of a customer and the
number of  customers in the system in transient and stationary regimes.

\section {Preliminaries  and  definitions }       

Let  $\varkappa,\,\delta \in \mathbb{N}=\{1,2,\dots\}$ be positive
independent  integer   random variables and  let $\eta\in(0,\infty)$ be
a positive  random variable independent of $\varkappa, \delta $ with
the  distribution function  $ F(x)=\bold P\left[\eta\le x\right],$ $
x\ge0.$
We will assume that
$\bold E \varkappa,$ $\bold E \delta,$ $\bold E\eta<\infty.$
 Introduce the sequences $\{\eta,\eta^{\prime}_{n}\},$
$\{\varkappa,\varkappa^{\prime}_{n}\},$
$\{\delta,\delta^{\prime}_{n}\},$ $n\in\mathbb{N} $ of independent
identically distributed (inside of each   sequence)  variables and
define the monotone sequences
\begin{align}                                       \label{dmg1B1}
&  \eta_0(x)=0,\quad \eta_{1}(x)=\eta_{x},\quad
   \eta_{n+1}(x)=\eta_{x}
   +\eta_{1}^{\prime}+\cdots+\eta_{n}^{\prime},
   \qquad n\in \mathbb{N},\\
&  \varkappa_{0}=0, \quad
   \varkappa_{n}=\varkappa^{\prime}_{1}+\cdots+\varkappa^{\prime}_{n};
   \qquad  \delta_{0}=0, \quad
   \delta_{n}=\delta^{\prime}_{1}+\cdots+\delta^{\prime}_{n}; \qquad
   n\in \mathbb{N},\notag
\end{align}
where $\eta_{x}\in(0,\infty) $  is a random variable with  the
following  distribution function
$$
   F_x(u)= \bold P\left[\eta_{x}\le u\right]=
   [F(x+u)-F(x)](1-F(x))^{-1}\qquad u\ge0.
$$
Denote  by  $\{\pi(t)\}_{t\ge 0}\in\mathbb {Z}^{+}=\{0,1,\dots\}$
a compound  Poisson process  with  the generating function of the form
$$
   \bold E\,\theta^{\pi(t)}=e^{tk(\theta)},\qquad
   k(\theta)=\mu\left(\bold E\,\theta^{\varkappa}-1\right),
   \quad |\theta|\le1,
$$
where  $\mu>0$   is the intensity of the jumps and $\varkappa$ is a
jump size.  For all $t\ge0 $  define  a renewal process generated by
the random sequence $\{\eta_{n}(x)\}_{n\in\mathbb{Z}^{+}}$ as follows
$$
   N_{x}(t)=\max\{n\in \mathbb{Z}^{+}:\eta_{n}(x)\le t\}
   \in\mathbb{Z}^{+},\qquad x\ge0.
$$
Introduce a  right-continuous step  process
for all  $ x\ge0 $
\begin{align}                                    \label{dmg1B2}
   D_{x}(t)= \pi(t)-\delta_{N_{x}(t)}
   \in\mathbb{Z}=\{0,\pm1,\dots \},
   \qquad t\ge0;\quad D_{x}(0)=0.
\end{align}
Note, that inter-arrival times of the positive jumps are
exponentially distributed with parameter $\mu,$  the positive jumps
themselves  are of a random size $\varkappa,$ and  there occur
negative jumps of size $\delta_{n}^{\prime} $ at time instants  $
\eta_{n}(x),  $ $n\in\mathbb{N}. $ We will call the process
$ \{D_{x}(t)\}_{t\ge0}$  a difference of the compound  Poisson process
and a compound  renewal process. Observe, that  this process is not a
Markov process in general. For all  $ t\ge0 $ introduce a
right-continuous linear  component
\begin{align}                                     \label{dmg1B3}
   \eta_{x}^{\,+}(t)= \left\{
   \begin{array}{l}
   t+x,\qquad \quad  0\le t< \eta_x, \\
   t-\eta_{N_{x}(t)}(x),\qquad t\ge\eta_x
   \end{array}
   \right. \in\mathbb R_{+}=[0,\infty),\qquad x\ge0.
\end{align}
The process  $ \left\{\eta_{x}^{\,+}(t)\right\}_{t\ge0}$ increases
linearly  on  the intervals  $ [\eta_{n}(x),\,\eta_{n+1}(x)),$
$n\in\mathbb{Z}^{+},$  it is killed to zero at the   points $
\eta_{n}(x),$ $n\in\mathbb{N},$  and the value of the process at  the
instant $ t_0\ge\eta_{x}$  is  equal to the time elapsed from the
moment of the last negative jump   of the process (\ref{dmg1B2})  till $
t_{0}.$ We will call the process (\ref{dmg1B3}) a linear component  (sometimes referred to as  the age process).
By adding this linear  component to the process
$ \left\{D_{x}(t)\right\}_{t\ge0}$
we  obtain a right-continuous
Markov process
\begin{align}                                         \label{dmg1B4}
   \{X_{t}\}_{t\ge0}=
   \left\{D_{x}(t),\,\eta_{x}^{\,+}(t) \right\}_{t\ge0}
   \in \mathbb{Z}\times\mathbb R_{+},
   \qquad  X_{0}=\{0,x\},\quad x\ge0,
\end{align}
which governs the process  $ \{D_{x}(t)\}_{t\ge0}.$ The process  defined in
(\ref{dmg1B4})  is a Markov process. Note, that it is   homogeneous
with respect to the first  component  \cite{ES}. This means  that if
$X_{t_0}=\{k,u\},$ $k\in\mathbb{Z},$ $u\ge0,$  then  the evolution
of the process $\{X_{t}\}_{t\ge t_0}$  in the sequel  does not
depend  on the value $k$ of the first component, and the first
positive jump of the process $\{D_{x}(t)\}_{t\ge t_0}$ (which is
distributed as  $\varkappa$) will occur   after  an exponential
period of time  with parameter  $\mu. $  The first  negative  jump
of the process  $\{D_{x}(t)\}_{t\ge t_0}$ (which is  distributed as
$\delta)$ will take  place after elapsing  of time  $\eta_u.$
This  fact  will
be used constantly when setting up the equations.

Here and in the sequel we assume that the random variable
$ \delta\in\mathbb{N}$ is geometrically distributed with parameter
$\lambda \in[0,1):$
\begin{align*}
   \bold P[\delta=n]=
   (1-\lambda)\lambda^{n-1},
   \quad n\in\mathbb{N},\qquad
   \bold E\theta^{\delta}=
   \theta\frac{1-\lambda}{1-\lambda\theta},\quad  |\theta|\le1.
\end{align*}
This assumption means that the process
$ \{D_{x}(t)\}_{t\ge 0} $ has geometrically distributed negative jumps
at  time instants $ \{\eta_{n}(x)\}_{n\in\mathbb{N}}.$
Throughout  the article  we will use the following notation
$\delta\sim ge(\lambda). $
In this  case it is possible to obtain closed form solutions for the one
and the two-sided boundary  problems.
Our task now is to determine  the Laplace transforms of the joint
distributions of  the  upper and lower one-boundary functionals of
the process $ \{X_{t}\}_{t\ge 0}.$ In the sequel we  will use the
following  result.

\begin{lemma}  \label{ldmg1B1}    

Let  $ \tilde f(s)=\bold Ee^{-s\eta}.$
Then for $s>0$ the equation
\begin{align}                                         \label{dmg1B5}
   \theta-\lambda=(1-\lambda)\tilde f(s-k(\theta)),
   \qquad |\theta|<1
\end{align}
has a unique solution $c(s)$ inside the circle  $|\theta|<1.$ This
solution is positive and  $c(s)\in(\lambda,1).$ If  $\bold
E[\varkappa], \bold E[ \eta]<\infty,$ $\rho= \mu(1-\lambda)\bold
E[\varkappa]\,\bold E[\eta],$ then   for  $\rho>1,$
$\lim_{s\to 0}\limits c(s)=c\in(\lambda,1);$ and for  $\rho\le 1,$
 $\lim_{s\to 0}\limits c(s)=1.$
\end{lemma}

A detailed proof of an analogous proposition for semi-continuous
random walks can be found in the monograph of Spitzer \cite{Sp1}.
The reasoning in that proof can  be applied to the equation (\ref{dmg1B5})
as well (see also Lemma 1 \cite{2Ka15}).

Let $ X_{0}=\{0,x\},$
$x\in\mathbb{R}_{+},$ $k\in\mathbb{Z}^{+}.$ Define
$$
   \tau_{k}(x)=\inf\{t:D_{x}(t)<-k\},\qquad
   T_{k}(x)=-D_{x}(\tau_{k}(x))-k,
   \qquad \inf\{\emptyset\}=\infty,
$$
i.e. the  first  overshoot  time of the negative level $-k $ by
the process  $\{D_{x}(t)\}_{t\ge 0}.$  We  will  use the convention
that on  the event $ \{\tau_{k}(x)=\infty\} $  $ T_{k}(x)=\infty.$
Denote $\mathfrak{B}_{k}(x)=\{\tau_{k}(x)<\infty\},$
$$
   f_{k}(x,m,s)=
   \bold E\left[e^{-s\tau_{k}(x)};
   T_{k}(x)=m,\mathfrak{B}_{k}(x) \right],\quad m\in\mathbb{N}.
$$
The Laplace transforms of the joint distribution of the lower
one-boundary functionals are  determined  by means of the following
lemma.

\begin{lemma}[\cite{2Ka13}]  \label{ldmg1B2}       

Let $ \{D_{x}(t)\}_{t\ge0}$ be  the difference of the  compound Poisson process
and the compound renewal process, $\delta\sim ge(\lambda).$ Then
\begin{itemize}
\item[{\rm(i)}]
the  Laplace transform of the  joint distribution of
$\{\tau_{k}(x),\,T_{k}(x)\},$ $k\in\mathbb{Z}^{+},$ $x\ge0$
satisfies  the following equality
for $s>0,$ $m\in\mathbb{N}$
\begin{align}                                   \label{dmg1B6}
   f_{k}(x,m,s)=
   \tilde f_x(s-k(c(s)))\,c(s)^{k}\,
   (1-\lambda)\lambda^{m-1},
\end{align}
where   $c(s)\in(\lambda,1)$ is the  unique solution of  the
equation (\ref{dmg1B5}) inside the circle  $|\theta|<1, $
$ \tilde f_x(s)=\bold E\,e^{-s\eta_x},$
$ \tilde f(s)= \bold E\,e^{-s\eta}=\tilde f_0(s);$
\item[{\rm(ii)}]
if  $\rho>1,$ then
$
   \bold P[\tau_{k}(x)<\infty]= \tilde f_x(-k(c))\,c^k<1,
$
and $\tau_{k}(x) $ for  all
$k\in\mathbb{Z}^{+},$ $x\ge0$ is  a defective  random variable;
if  $\rho\le 1,$
then $\bold P[\tau_{k}(x)<\infty]=1,$ and  $\tau_{k}(x)$   is  a
proper  variable for all $k\in\mathbb{Z}^{+},$ $x\ge0$.
\end{itemize}
\end{lemma}
Observe that for all $k\in\mathbb{Z} $ the value of the overshoot
$ T_{k}(x) $  does not depend on  $\tau_{k}(x) $ and it is geometrically
distributed $ T_{k}(x)\sim ge(\lambda). $

We  now introduce  a sequence   which  will be used  to obtain  the
results  in the  sequel. The idea to employ this sequence for
semi-continuous random walks and semi-continuous L\'evy  processes
is due to Tak\'{a}cs \cite{Tac1}. Since the function
$$
  \tilde f_x(s-k(\theta))
  =\bold E\left[e^{-s\eta_x}\theta^{\pi(\eta_x)}\right]
  = \sum_{i\in\mathbb{Z}^{+}}\theta^{i}\int_0^{\infty}e^{-st}
  \bold P[\eta_x\in dt,\,\pi(t)=i],\quad |\theta|\le1,
$$
is analytic inside the  unit circle for all
$ s,x\ge0 $, then  the function
\begin{align}                                 \label{dmg1B7}
   \mathbb Q_{\theta}^{s}(x)
   = \frac{(1-\lambda)\tilde f_x(s-k(\theta))}
   {(1-\lambda)\tilde f(s-k(\theta))
   +\lambda-\theta},
   \qquad s,x\ge0,\quad |\theta|<c(s)
\end{align}
is analytic on the open set   $|\theta|<c(s).$ In this region it
can be represented as a powers series
$$
  \mathbb Q_{\theta}^{s}(x)=
  \sum_{k\in\mathbb{Z}^{+}}\theta^{k}Q_{k}^{s}(x),
  \qquad s,x\ge0,\quad |\theta|<c(s).
$$
The coefficients of this expansion can be calculated
by means of the inversion formula
\begin{align}                                       \label{dmg1B8}
   Q_{k}^{s}(x)
   =\frac{1}{2\pi i}\oint_{|\theta|
   =\alpha}\,\frac1{\theta^{k+1}}\,
   \frac{(1-\lambda)\tilde f_x(s-k(\theta))}{(1-\lambda)
   \tilde f(s-k(\theta))+\lambda-\theta}\,d\theta,
   \quad \alpha\in(0,c(s)).
\end{align}
We will call the sequence  $\{Q_{k}^{s}(x)\}_{k\in\mathbb{Z}^{+}},$
$x\ge0,$  defined by the formula (\ref{dmg1B8})  the resolvent sequence of
the process $\{D_{x}(t)\}_{t\ge0}.$

We now  explain a   probabilistic  meaning   of this sequence.
Introduce a random sequence as follows:  (see \cite{3Ka1})
\begin{align*}
    X_{0}(x)=0,\quad  X_{1}(x)=\pi(\eta_{x})-\delta,\quad
    X_{n+1}(x)=X_{1}(x)+\sum_{i=1}^{n}X_{i}^{\prime},
    \quad X_{n}=X_{n}(0),
\end{align*}
where $ X=\pi(\eta)-\delta\in\mathbb{Z},$ $ \{X,X_{n}^{\prime}\},$
$ n\in\mathbb{N} $ is a sequence of i.i.d. random variables. We now define a
right-continuous step process in the following way:
\begin{align*}
   \{S_{x}(t)\}_{t\ge0}=
   \left\{X_{N_{x}(t)}(x) \right\}_{t\ge0}
   \in \mathbb{Z}, \qquad  S_{x}(0)=0,
   \quad x\in\mathbb{R}_{+}.
\end{align*}
The sample paths of the process are constant on the time intervals
$ [\eta_{n}(x),\eta_{n+1}(x)),$ $n\in\mathbb{Z}^{+}$ and there occur jumps at
the instants $\eta_{n}(x),$ $n\in\mathbb{N}.$ These jumps have the same
distribution as $X\doteq\pi(\eta)-\delta,$ where $ n\in\{2,3, \dots \},$ and
$X_{1}(x)\doteq \pi(\eta_{x})-\delta$ for $n=1.$ Here and in the sequel we
will call the process  $\{S_{x}(t)\}_{t\ge0}$ a semi-Markov random walk
generated by the sequences $\{\eta_{n}(x)\},$ $\{X_{n}(x)\},$
$n\in\mathbb{Z}^{+}.$ Let
$
   S^{+}_{t}=\sup_{u\le t} S_{0}(u)
$
be the supremum $\{S_{0}(t)\}_{t\ge0}.$ The  generating  function of
$S^{+}_{t}$  was found in \cite{3Ka1}:
\begin{align*}
   \bold E\theta^{S_{\nu_s}^{+}}=
   \frac{1-\lambda}{1-c(s)}\,
   \frac{(1-\tilde f(s))(\theta-c(s))}
   {\theta-\lambda-(1-\lambda)
   \tilde f(s-k(\theta)},\qquad |\theta|\le1,
\end{align*}
where $ \nu_{s} $ is an exponential variable with parameter $ s>0, $
independent from  the process
$\{S_{x}(t)\}_{t\ge0}$. It follows from  (\ref{dmg1B7}) and from the
latter formula    that for  $ |\theta|<c(s) $
\begin{align*}
   \mathbb Q_{\theta}^{s}(x)=\frac{1-\tilde f(s)}{1-c(s)}\,
   \frac{\tilde f_x(s-k(\theta))}{c(s)-\theta}\,\bold E\theta^{S_{\nu_s}^{+}},
   \qquad |\theta|<c(s).
\end{align*}
Comparing  the coefficients  of $ \theta^{k}, $ $ k\in\mathbb{Z}^{+}$ in both
sides yields
\begin{align*}
   Q_{k}^{s}(x)=\frac{1-c(s)}{1-\tilde f(s)}\,
   \sum_{i=0}^{k}c(s)^{i-k-1}\sum_{j=0}^{i}\bold E\left[e^{-s\eta_{x}},\pi(\eta_x)=j\right]
   \bold P[S_{\nu_s}^{+}=i-j].
\end{align*}
Denote by $ \pi^{s}(\eta_x)\in\mathbb{Z}^{+}, $ $s>0 $  a random variable
given by  its distribution:
$$
    \bold P\left[\pi^{s}(\eta_x)=k\right]=\frac{1}{\tilde f_x(s)}
    \bold E\left[e^{-s\eta_{x}},\pi(\eta_x)=k\right],\qquad
    k\in\mathbb{Z}^{+}.
$$
Then the previous equality implies that
\begin{align*}
   Q_{k}^{s}(x)=\frac{\tilde f_x(s)}{1-\tilde f(s)}\,\frac{1-c(s)}
   {c(s)^{k+1}}\,
   \sum_{i=0}^{k}c(s)^{i} \bold P[\pi^{s}(\eta_x)+S_{\nu_s}^{+}=i],
  \qquad k\in\mathbb{Z}^{+}.
\end{align*}
which explains the probabilistic  meaning of the resolvent  sequence.
Asymptotically,  one has that $ Q_{k}^{s}(x)\sim c(s)^{-k} $ as $ k\to\infty.$

Let $ X_{0}=\{0,x\},$  $x\ge0,$ $k\in\mathbb{Z}^{+} $ and introduce
upper  one-boundary functionals of the process $ \{X_{t}\}_{t\ge0}:$
\begin{align*}
   \tau^{k}(x)=\inf\{t:\,D_{x}(t)>k\},\quad
   T^{k}(x)=D_{x}(\tau^{k}(x))-k,\quad\eta^{k}(x)
   =\eta_{x}^{\,+}(\tau^{k}(x))
\end{align*}
i.e.  the instant of the  first  crossing of   the level $k$ by the
process $\{D_{x}(t)\}_{t\ge 0},$ the value  of the overshoot across
the upper level   and the value of the linear component
$\eta_{x}^{+}(\cdot)$  at the instant of the first crossing
(the time since the last renewal).
Denote $\mathfrak{B}^{k}(x)=\{\tau^{k}(x)<\infty\},$
$$
   f^{k}(x,dl,m,s)=
   \bold E\left[e^{-s\tau^{k}(x)};\eta^{k}(x)\in dl,
   T^{k}(x)=m,\mathfrak{B}^{k}(x) \right],\quad m\in\mathbb{N}.
$$
We now determine the upper one-boundary functionals of the process
$\{D_{x}(t)\}_{t\ge0}.$ Let  $ k\in\mathbb{Z}^{+}  $  and
$
  \tilde\tau^k=\inf\{t:\,\pi(t)>k\},$
  $ \tilde T^k=\pi(\tilde\tau^k)-k
$
be the  first  crossing  time through the upper level $k$ by the
compound Poisson process $\{\pi(t)\}_{t\ge0}$ and the value of the overshoot
at this instant. Denote by
\begin{align*}
&  \rho_k(t)=\bold P[\pi(t)=k],\quad
   \sum_{k=0}^{\infty}\theta^{k}\rho_k(t)=
   \bold E\,\theta^{\pi(t)}=e^{tk(\theta)},\quad |\theta|\le1,\\
&  p_{k}^{m}(dt)=\bold P[\tilde\tau^k\in dt,\,\tilde T^k=m]=
   \mu\sum_{i=0}^{k}\rho_i(t)\bold P[\varkappa=k-i+m]\,dt,
   \quad m\in\mathbb{N}.
\end{align*}

\begin{lemma} [\cite{2Ka13}]                \label{ldmg1B3}  

Let $ \{D_{x}(t)\}_{t\ge0}$ be the  difference of the compound Poisson
process and the compound renewal process, $\delta\sim ge(\lambda),$
$x\ge0,$ $ k\in\mathbb{Z}^{+} $ and
$\{Q_{k}^{s}(x)\}_{k\in\mathbb{Z}^{+}},$  be the  resolvent
sequence of the process  $\{D_{x}(t)\}_{t\ge 0},$  given by (\ref{dmg1B8}).
\begin{itemize}
\item[{\rm(i)}]
the Laplace transforms of the joint distribution of
$ \{\tau^{k}(x), \eta^{k}(x), T^{k}(x)\} $
satisfy the  following equality
\begin{align}                                         \label{dmg1B9}
&  f^{k}(x,dl,m,s)
   = e^{-s(l-x)}\,\frac{1-F(l)}{1-F(x)}\bold I\{l>x\}
   p_{k}^{m}(d(l-x))\notag\\
&  + \Phi^{s}_{\lambda}(0,dl,m)\,Q_{k}^{s}(x)
   -e^{-sl}\,[1-F(l)]\sum_{i=0}^{k}Q_{i}^{s}(x)\,
   p_{k-i}^{m}(dl),
\end{align}
where
$ \Phi^{s}_{\lambda}(0,dl,m)=
e^{-sl}[1-F(l)]\sum_{k=0}^{\infty}c(s)^{k}p_{k}^{m}(dl);$
\item[{\rm(ii)}]
  the Laplace transform  of the first crossing time  through
the upper level  $k$ by the process
$\{D_{x}(t)\}_{t\ge 0}$ are such that
for all $ k\in\mathbb{Z}^{+},$ $ s,x\ge0 $
\begin{align}                                            \label{dmg1B10}
   f^{k}(x,s)= \bold E e^{-s\tau^{k}(x)}=
   1-\frac{s}{s-k(c(s))}\frac{Q_{k}^{s}(x)}{1-\lambda}- A_{x}^{k}(s),
\end{align}
where
$
   A_{x}^{k}(s)= \sum_{i=0}^{k}\limits\tilde\rho_i(s)
   \left[1-Q_{k-i}^{s}(x)(1-\lambda)^{-1}\right],$
   $\tilde\rho_k(s) =s\int_{0}^{\infty}e^{-st}\rho_{k}(t)\,dt;
$
\item[{\rm(iii)}]
for  $\bold E[\varkappa],\,\bold E[\eta]<\infty$ and  $\rho<1,$
$ \tau^{k}(x) $ is  a defective random variable and
$$
   \bold P[\tau^{k}(x)<\infty]=1-
   (1-\rho)(1-\lambda)^{-1}Q_{k}(x) <1,\qquad
   k\in\mathbb{Z}^{+} \quad x\ge0,
$$
where  $\{Q_{k}(x)\}_{k\in\mathbb{Z}^{+}},$ $x\ge 0 $ is the
resolvent sequence of the process
$\{D_{x}(t)\}_{t\ge 0},$ given by
(\ref{dmg1B8})  for $s=0:$
\begin{align}                                         \label{dmg1B11}
   Q_{k}(x)=\frac{1}{2\pi i}\oint_{|\theta|=\alpha}
   \frac{d\theta}{\theta^{k+1}}\, \frac{(1-\lambda)
   \tilde f_{x}(-k(\theta))}{(1-\lambda)
   \tilde f(-k(\theta))+\lambda-\theta},
   \quad \alpha\in(0,c(0));
\end{align}
if $\rho\ge1,$ then for all $ k\in\mathbb{Z}^{+},$ $ x\ge0 $
$\tau^{k}(x)$ is a  proper  random variable.
\end{itemize}
\end{lemma}
Along with expression (\ref{dmg1B11}) there exists another  way to  calculate
$ Q_{k}(x),$ which is more applicable from practical point of view.  We will
now derive  the recurrent  formula  for $ Q_{k}(x). $
It follows  from   (\ref{dmg1B7}) for  $   s, \theta=0 $ that
\begin{align*}
    Q_{0}(x)=
    (1-\lambda)
    (\lambda+(1-\lambda)f_{0})^{-1} f_{0}(x),
\end{align*}
where  for all  $ k\in\mathbb{Z}^{+}$
$$
  f_{k}(x)
  =\bold P\left[\pi(\eta_x)=k\right]
  = \int_0^{\infty}
  \bold P[\eta_x\in dt,\,\pi(t)=k],\quad f_{k}=f_{k}(0).
$$
Again, it follows from  (\ref{dmg1B7})  for $ s=0 $  that
\begin{align*}
   (1-\lambda)\tilde f_x(-k(\theta))=
   (1-\lambda)\tilde f(-k(\theta))\mathbb Q_{\theta}(x)+
   (\lambda-\theta)\mathbb Q_{\theta}(x).
\end{align*}
Comparing the coefficients  of  $ \theta^{k}, $  $ k\in\mathbb{N} $
in both sides  implies that
\begin{align*}
   (1-\lambda)f_{k}(x)=
   (1-\lambda)\sum_{i=0}^{k}Q_{i}(x)f_{k-i}+
   \lambda Q_{k}(x) - Q_{k-1}(x).
\end{align*}
Combining like terms  yields
\begin{align*}
   \left(\lambda+(1-\lambda)f_{0} \right)Q_{k}(x)=
   (1-\lambda)f_{k}(x)+ Q_{k-1}(x)-
   (1-\lambda)\sum_{i=0}^{k-1} Q_{i}(x)f_{k-i}.
\end{align*}
The  latter  formula is a  recurrent  relation  which allows  to calculate
successively the terms $ Q_{k}(x) $ given the previous terms
$ Q_{0}(x),\dots,Q_{k-1}(x).$ For instance,  given the expression for
$ Q_{0}(x) $  one finds  that
$$
   Q_{1}(x) =\frac{1-\lambda}{\lambda+(1-\lambda)f_{0}}
   \left[f_{1}(x)+\frac{1-(1-\lambda)f_{0}}
   {\lambda+(1-\lambda)f_{0}}\;f_{0}(x)  \right].
$$

The  knowledge of the one-boundary characteristics of the process
allows us to  solve the two-sided problems, which is the  aim
of the   following section.

\section {Two-sided  problems for the     
process  $\{D_{x}(t)\}_{t\ge 0}$ }            

Let  $B\in \mathbb{Z}^{+}$ be fixed, $k\in\overline{0,B},$ $r=B-k,$
$X_{0}=\{0,x\},$ $ x\ge0,$  and introduce  the random variable
$$
  \chi_{r}^{B}(x)=\inf\{t:\,D_{x}(t)\notin[-r,k]\}\stackrel{\rm def}{=} \chi
$$
the first exit time from the interval $ [-r,k] $ by the process
$\{D_{x}(t)\}_{t\ge 0}.$ This random variable takes values from a
countable  set  $ \{\xi_{n},\,n\in\mathbb{N}\}
\cup\{\,\eta_{n}(x),\,n\in\mathbb{N}\},$ and it  is a Markov time
of the process  $\{X_{t}\}_{t\ge 0}$  ($\xi_k$ are the  instants of  jumps of the process $\pi(t).$)  Note,  that the  exit from the interval  can
occur  either    through the  upper   boundary  $k,$ or through the
lower boundary $-r.$   In  view of this remark introduce the events
\\$\mathfrak{A}^{k}=\{D_{x}(\chi)>k\},$ i.e. the process
$\{D_{x}(t)\}_{t\ge 0}$ exits the interval  $[-r,k]$ through the
upper  boundary $k$;
\\$\mathfrak{A}_{r}=\{D_{x}(\chi)<-r\},$ i.e. the process
$\{D_{x}(t)\}_{t\ge 0}$ exits the interval
$[-r,k]$ through the lower boundary $-r.$  Denote by
$$
   T=(D_{x}(\chi)-k)\bold I_{\mathfrak{A}^{k}}+(-D_{x}(\chi)-r)
   \bold I_{\mathfrak{A}_{r}},\quad L=\eta_{x}^{\,+}(\chi)
   \bold I_{\mathfrak{A}^{k}}+0\cdot
   \bold I_{\mathfrak{A}_{r}}, \quad
   \bold P[\mathfrak{A}^{k}+\mathfrak{A}_{r}]=1
$$
the value  of the overshoot   through the boundaries  of the
interval $[-r,k] $   by the process  $\{D_{x}(t)\}_{t\ge 0} $ and
the value of the  linear component  at the instant of the first exit
(the time since the last renewal),
where  $\bold I_{\mathfrak{A}}=\bold I_{\mathfrak{A}}(\omega)$
is the indicator function of the event $\mathfrak{A}.$ Denote
\begin{align*}
  V^{k}(x,dl,m,s)
  =\bold E\left[e^{-s\chi};L\in dl,T=m,\mathfrak{A}^{k}\right],\;
  V_{r}(x,m,s)
  =\bold E\left[e^{-s\chi};T=m,\mathfrak{A}_{r}\right].
\end{align*}

\begin{theorem} [\cite{2Ka13}]   \label{tdmg1B1}     

Let  $\{D_{x}(t)\}_{t\ge 0}$ be the difference of the compound Poisson process
and  the  renewal process (2),
$\delta\sim ge(\lambda),$  $\{Q_{k}^{s}(x)\}_{k\in\mathbb{Z}^{+}},$
be the  resolvent sequence  of the process given by  (\ref{dmg1B8}),
$ Q_{k}^{s}\stackrel{\rm def}{=} Q_{k}^{s}(0).$
Then
\begin{itemize}
\item[{\rm(i)}]
the Laplace transforms of the joint distribution of  $\{\chi,L,T\}$
satisfy   the following equalities  for all
$x,s\ge0,$ $m\in\mathbb{N}$
\begin{align}                                     \label{dmg1B12}
&  V_{r}(x,m,s)=
   \frac{Q_{k}^{s}(x)}{\bold E\,Q_{\delta+B}^{s}}
   (1-\lambda)\lambda^{m-1},\notag\\
&  V^{k}(x,dl,m,s)= f^{k}(x,dl,m,s)-
   \frac{Q_{k}^{s}(x)}
   {\bold E\,Q_{\delta+B}^{s}} \,
   \bold  E f^{\delta+B}(0,dl,m,s),
\end{align}
where  the function $ f^{k}(x,dl,m,s) $ is given by (\ref{dmg1B9}),
\begin{align*}
&  \bold E\,Q_{\delta+B}^{s}=
   \sum_{k\in\mathbb{N}}(1-\lambda)
   \lambda^{k-1} Q_{k+B}^{s},\\
&  \bold E f^{\delta+B}(0,dl,m,s)=
   \sum_{k\in\mathbb{N}}(1-\lambda)
   \lambda^{k-1} f^{k+B}(0,dl,m,s);
\end{align*}
\item[{\rm(ii)}]
for the Laplace transforms of the first exit time $\chi$  from
the interval by  the  process
$\{D_{x}(t)\}_{t\ge 0}$  the
following formulae hold
\begin{align}                                        \label{dmg1B13}
&  \bold E\left[e^{-s\chi};\mathfrak{A}_{r}\right]=
   \frac{Q_{k}^{s}(x)}
   {\bold E\,Q_{\delta+B}^{s}},\\
&  \bold E\left[e^{-s\chi};\mathfrak{A}^{k}\right]=
   1- A^{k}_{x}(s)-
   \frac{Q_{k}^{s}(x)} {\bold E\,Q_{\delta+B}^{s}}
   \left(1- \bold E\,A^{\delta+B}_{0}(s) \right),\notag
\end{align}
where
$
    \bold E\,A^{\delta+B}_{0}(s)= \sum_{k\in\mathbb{N}}\limits(1-\lambda)
   \lambda^{k-1} A_{0}^{k+B}(s);
$
\item[{\rm(iii)}]
the probabilities of the exit from the interval through the
upper and the lower boundary  by the process $\{D_{x}(t)\}_{t\ge 0}$
are given by
$$
   \bold P[\mathfrak{A}_{r}]=
   \frac{Q_{k}(x)}{\bold E\,Q_{\delta+B}},\qquad
   \bold P[\mathfrak{A}^{k}]=
   1-\frac{Q_{k}(x)}{\bold E\,Q_{\delta+B}},
$$
where the resolvent sequence of the process
$\{Q_{k}(x)\}_{k\in\mathbb{Z}^{+}},$ $x\ge0$ is defined by  (\ref{dmg1B11}),
$ Q_{k}\stackrel{\rm def}{=} Q_{k}(0).$
\end{itemize}
\end{theorem}

Denote  by  $ \nu_{s}\sim\exp(s) $ an  exponential random variable   with parameter
$ s>0 $  independent of the  process  $ D_{x}(t).$  For
$k\in\mathbb{Z}^{+}, $ $x\ge0 $   define
$
    D_{x}^{+}(t)=\sup_{[0,t]}\limits D_{x}(\cdot),$ $
    D_{x}^{-}(t)=\inf_{[0,t]}\limits D_{x}(\cdot)
$
the  running   maximum and minimum  of   the process.
Our  aim is to  determine the  joint distribution of
$\{D_{x}^{-}(t),D_{x}(t), D_{x}^{+}(t)\}.$  In order to  do this, we will
require the  joint  distribution  of
$ \{D_{x}(\nu_{s}),D_{x}^{+}(\nu_{s})\}.$

\begin{lemma}  \label{ldmg1B4}
Let  $k\in\mathbb{Z}^{+}$ and
$
    E_{k}^{+}(x,z,s)=\bold E\left[z^{D_{x}(\nu_{s})};
    D_{x}^{+}(\nu_{s})\le k \right],$ $ |z|\ge 1
$
be the generating function  of the joint distribution of
$ \{D_{x}(\nu_{s}),D_{x}^{+}(\nu_{s})\}. $
Then
\begin{itemize}
\item[\rm{(i)}]
the generating function   $ E_{k}^{+}(x,z,s)$
is  such that
\begin{align}                                         \label{dmg1B14}
    E_{k}^{+}(x,z,s)=z^{k} A_{x}^{k}(s)+
    (1-z)\sum_{i=0}^{k-1}z^{i}A_{x}^{i}(s)+
    z^{k}Q_{k}^{s}(x)\mathbb{E}_{\lambda/z}^{s}(0,z),
\end{align}
where
$$
   \mathbb{E}_{\lambda/z}^{s}(0,z)=\frac{s(1-\lambda)^{-1}}{s-k(c(s))}\,
   \frac{1-c(s)}{1-c(s)/z};
$$
\item[\rm{(ii)}]
the joint distribution
$
    \mathfrak{E}_{k}^{+}(x,u,s) =
   \bold P\left[D_{x}(\nu_{s})\le u,
    D_{x}^{+}(\nu_{s})\le k \right],
$
$ u\in\overline{-\infty,k} $
satisfies the following equality
\begin{align}                                        \label{dmg1B15}
    \mathfrak{E}_{k}^{+}(x,u,s) =
    A_{x}^{u}(s)+\frac{s(1-\lambda)^{-1}}{s-k(c(s))}\,c(s)^{k-u}
    Q_{k}^{s}(x),\quad   A_{x}^{u}(s)=0,\;  u<0;
\end{align}
\item[\rm{(iii)}]
under  the   condition  $(A)$
\begin{itemize}
\item[{\rm(A)}]
$  \rho=(1-\lambda)\mu\bold E\eta\bold E \varkappa=1,\quad
   \sigma^{2}=\mu\left[\bold E\varkappa(\varkappa-1)+
   \frac{\bold E\varkappa\bold E\eta^{2}}
   {(1-\lambda)(\bold E\eta)^{2}}\right]<\infty,$
\end{itemize}
the  following limiting equality holds as $ B\to\infty, $ $k>0 $
\begin{align*}
     \bold P[D_{x}(tB^{2})&\le [uB],
     D_{x}^{+}(tB^{2})\le [kB] ]
     \to \frac{1}{\sigma\sqrt{2\pi t}}
     \int_{-u}^{2k-u}e^{-v^{2}/2\sigma^{2}t}dv,
\end{align*}
where  $[a] $ is the integer part of the number $ a,$
$ u \le k.$
\end{itemize}
\end{lemma}

\begin{proof}
In view of the   total probability law,  homogeneity of the process  $X_{t}$
with respect to the first component, Markov property of   $ \eta_{1}(x) $
we can write for the function $ E_{k}^{+}(x,z,s),$ $k\in\mathbb{Z}^{+},$
$x\ge0 $ the following  equation
\begin{align}                                          \label{dmg1B16}
&   E_{k}^{+}(x,z,s)=
    s\int_{0}^{\infty}e^{-st}\bold P\left[\eta_{x}>t\right]
    \bold E \left[z^{\pi(t)};\pi(t)\le k \right]dt+\notag\\
&   + \int_{0}^{\infty}e^{-su}\sum_{v=0}^{k}
    \bold P\left[\eta_{x}\in du,\pi(u)=v \right]
    z^{v}\sum_{r=1}^{\infty}(1-\lambda)
    \lambda^{r-1}z^{-r} E_{k-v+r}^{+}(0,z,s).
\end{align}
Introduce the generating function
$
    \mathbb{E}_{\theta}^{s}(x,z)=
    \sum_{k\in\mathbb{Z}^{+}} \limits \theta^{k}E_{k}^{+}(x,z,s), $
     $ |\theta|<1.
$
Multiplying  (\ref{dmg1B16}) by  $\theta^{k} $ and summing  over
$k\in\mathbb{Z}^{+},$  we derive  the  following equation  for the function
$ \mathbb{E}_{\theta}^{s}(x,z) $
\begin{align}                                      \label{dmg1B17}
     \mathbb{E}_{\theta}^{s}(x,z)
&    = \frac{s}{1-\theta}\,
     \frac{1-\tilde f_{x}(s-k(z\theta))}{s-k(z\theta)}+\\
&    + \tilde f_{x}(s-k(z\theta))\frac{(1-\lambda) }{\lambda-z\theta}\,
    \left[\mathbb{E}_{\lambda/z}^{s}(0,z)
    -\mathbb{E}_{\theta}^{s}(0,z) \right],\quad |\theta|<1,\; |z|\ge 1. \notag
\end{align}
Letting  $ x=0 $ in the latter  equation  yields
\begin{align*}
      \mathbb{E}_{\theta}^{s}(0,z)
&     = \frac{z\theta-\lambda}
      {(1-\lambda) \tilde f(s-k(z\theta))+\lambda-z\theta}\times\\
&     \times\left[       \tilde f(s-k(z\theta))\frac{1-\lambda}
      {z\theta -\lambda}\,\mathbb{E}_{\lambda/z}^{s}(0,z)-
      \frac{s}{1-\theta}
      \frac{1-\tilde f(s-k(z\theta))}{s-k(z\theta)} \right].
\end{align*}
The  function which  enters  the  left-hand side of this equation is
analytic in $ |\theta|<1.$ In view  of Lemma \ref{ldmg1B1}  it has
denominator  of the right-hand side  has a  simple   zero in $ \theta=c(s)/z.$
Hence, the nominator of   right-hand side
should also have the simple zero. Letting $ \theta=c(s)/z $ in the nominator
we find the function $ \mathbb{E}_{\lambda/z}^{s}(0,z)$
$$
     \mathbb{E}_{\lambda/z}^{s}(0,z)=
     \frac{s(1-\lambda)^{-1}}{s-k(c(s))}\,\frac{1-c(s)}{1-c(s)/z},
     \qquad |z|\ge 1.
$$
Employing the  definition of the resolvent (\ref{dmg1B7}) and substituting
the expression for $ \mathbb{E}_{\theta}^{s}(0,z) $ into (\ref{dmg1B17}),
we get
\begin{align}                                      \label{dmg1B18}
     \mathbb{E}_{\theta}^{s}(x,z)=
     \mathbb{A}^{z\theta}_{x}(s)+
     \theta\,\frac{1-z}{1-\theta}\,\mathbb{A}^{z\theta}_{x}(s)
     + \mathbb{Q}_{z\theta}^{s}(x)\mathbb{E}_{\lambda/z}^{s}(0,z),
\end{align}
where
$$
     \mathbb{A}^{z\theta}_{x}(s)=\sum_{k=0}^{\infty}(z\theta)^{k}
     A^{k}_{x}(s)=\frac{s}{s-k(z\theta)}
     \left(\frac{1}{1-z\theta}-\frac{1}{1-\lambda}\,
     \mathbb{Q}_{z\theta}^{s}(x) \right).
$$
Using the definition of the resolvent (\ref{dmg1B8}) and comparing
the coefficients of $\theta^{k},$ $k\in\mathbb{Z}^{+} $ in both
sides of  (\ref{dmg1B18}) implies that
\begin{align*}
    E_{k}^{+}(x,z,s)=z^{k} A_{x}^{k}(s)+
    (1-z)\sum_{i=0}^{k-1}z^{i}A_{x}^{i}(s)+
    z^{k}Q_{k}^{s}(x)\mathbb{E}_{\lambda/z}^{s}(0,z),
\end{align*}
i.e.  the  equality (\ref{dmg1B14}) of the lemma.
Comparing the coefficients of
$ z^{i},$ $ i\in\overline {-\infty,k} $ in both sides of the latter equality,
we find
\begin{align*}
    \bold P[D_{x}(\nu_{s})&=i,
    \;D_{x}^{+}(\nu_{s})\le k ]= \\
&   = A_{x}^{i}(s)-A_{x}^{i-1}(s)+
   \frac{s}{s-k(c(s))}\,\frac{1-c(s)}{1-\lambda}\, c(s)^{k-i}
    Q_{k}^{s}(x), \qquad i\le k,
\end{align*}
where $   A_{x}^{i}(s)=0,$ for $ i<0.$
The latter    formula  implies  (\ref{dmg1B15}).
Denote
$
   \tilde e_{k}^{t}(x,u,B)= \bold P\left[D_{x}(tB^{2})\le [uB],
    D_{x}^{+}(tB^{2})\le [kB] \right].
$
It is  clear that
$$
   \lim_{B\to\infty} \int_{0}^{\infty}e^{-st}\tilde e_{k}^{t}(x,u,B)\,dt=
     \frac{1}{s}\lim_{B\to\infty} \mathfrak{E}_{[kB]}^{s/B^{2}}(x,[uB]).
$$
To proceed  further, we  need the  following
limiting   equalities (see \cite{3Ka1})
\begin{align}                                      \label{dmg1B19}
&    c(s/B^{2})=
     1-{B}^{-1}\sqrt{2s}/\sigma+o(B^{-1}),\notag\\
&   \lim_{B\to\infty} B^{-1}
    Q_{[kB]}^{s/B^{2}}(x)=
    \frac{2\sh\left(k\sqrt{2s}/\sigma\right)}
    { \sigma\sqrt{2s}\bold E\eta}=
    \lim_{B\to\infty} B^{-1}\bold E Q_{\delta+[kB]}^{s/B^{2}},\notag\\
&    \lim_{B\to\infty} A^{[kB]}_{x}(s/B^{2})= 1-
    \ch\left(k\sqrt{2s}/\sigma\right) =
    \lim_{B\to\infty} A^{\delta+[kB]}_{0}(s/B^{2}).
\end{align}
In view of these equalities and of the formula  (\ref{dmg1B15}) we  derive
\begin{align*}
    \lim_{B\to\infty}
&   \int_{0}^{\infty}e^{-st}\tilde e_{k}^{t}(x,u,B)\,dt=
    {s}^{-1} \bold I_{\{u<0\}}\left(e^{u\sqrt{2s}/\sigma}/2
    -e^{-(2k-u)\sqrt{2s}/\sigma}/2\right)\\
&   + {s}^{-1} \bold I_{\{u\in[0,k]\}}
    \left(1-e^{-u\sqrt{2s}/\sigma}/2-e^{-(2k-u)\sqrt{2s}/\sigma}/2\right),
    \quad u\le k.
\end{align*}
Denote  by  $w_{\{t\ge0\}}$  the  symmetric Wiener process with the dispersion
$\sigma $ and  by $ \tau^{a}=\inf\{t:w_{t}\ge a\}$ the first passage time of
the  level $ a\in\mathbb{R}_{+}.$  The L\'{e}vy  formula
$ \bold P\left[\tau\le t\right]= 2\bold P\left[w_{t}\ge a\right] $
implies for the Laplace transforms  that
$$
    \frac{1}{s}\,e^{-a\sqrt{2s}/\sigma}=
    2\int_{0}^{\infty}e^{-st} \bold P\left[w_{t}\ge a\right]dt.
$$
Employing  the latter  formula  to  invert the Laplace  transforms
in the previous  equality, we derive the second limiting formula of the
theorem.
\end{proof}

Let  $ k,r\in\mathbb{Z}^{+},$  $u\in\overline{-r,k} $ and denote by
\begin{align*}
    \tilde e_{r,k}^{t}(x,u)
&   =\bold P \left[-r\le D_{x}^{-}(t),\,
    D_{x}^{-}(t) \le u,\, D_{x}^{+}(t)\le k \right]=
    \bold P \left[ D_{x}(t)\le u,\, \chi_{x}^{B}(r)>t\right], \\
&   \mathfrak{E}_{r,k}^{s}(x,u)= \tilde e_{r,k}^{\nu_{s}}(x,u)=
    s\int_{0}^{\infty} e^{-st} \tilde e_{r,k}^{t}(x,u)\,dt
\end{align*}
the joint distribution of $ \{D_{x}^{-}(t),D_{x}(t),D_{x}^{+}(t)\}$
and its  Laplace  transform.

\begin{theorem}              \label{tdmg1B2}

Let $ \nu_{s}\sim\exp(s) $  be an exponential random variable independent
of the process $ D_{x}(t),$
$B=r+k.$ Then
\begin{itemize}
\item[\rm{(i)}]
the joint distribution of
$ \{D_{x}^{-}(\nu_{s}),D_{x}(\nu_{s}),D_{x}^{+}(\nu_{s})\}$
is such  that
\begin{align}                                         \label{dmg1B20}
    \mathfrak{E}_{r,k}^{s}(x,u)=
    A_{x}^{u}(s) -
    \frac{Q_{k}^{s}(x)}{\bold E\,Q_{\delta+B}^{s}}\,
    \bold E A^{\delta+r+u}_{x}(s), \quad u\in\overline{-r,k},
\end{align}
where
$
\bold E A^{\delta+r+u}_{x}(s)=
(1-\lambda)\sum_{i=1}^{\infty}\limits \lambda^{i-1}A_{x}^{i+r+u}(s);$
\item[\rm{(ii)}]
under the condition $ (A) $ and $r\in(0,1),$ $k=1-r$ the
joint distribution
$ \tilde e_{[rB],[kB]}^{tB^{2}}(x,[uB]) $  weakly converges as
$ B\to\infty $ to the joint  distribution
$$
    \bold P\left[-r\le \inf_{v\le t} w_{v},\, w_{t}\le u,\,
    \sup_{v\le t} w_{v} \le k \right], \quad u\in[-r,k]
$$
of the infimum, the supremum and the value of the symmetric
Wiener process with  the  dispersion  \textbf{$ \sigma. $}
In addition,  the following
limiting equality holds
\begin{align}                                       \label{dmg1B21}
    \lim_{B\to\infty}
    \tilde e_{[rB],[kB]}^{tB^{2}}(x,[uB])=
    \frac{4}{\pi}
    \sum_{n\in\mathbb{N}} \frac{e^{-\frac{t}{2}(\pi n\sigma)^{2}}}{n}\,
    \sin(r\pi n)
    \sin^{2}\left(\frac{r+u}{2}\,n\pi\right).
\end{align}
\end{itemize}
\end{theorem}

\begin{proof}
The total probability law, homogeneity  of the process $ X_{t} $
with respect  to the  first  component, Markov property of
$ \chi_{r}^{B}(x) $  for  all $ k,r\in\mathbb{Z}^{+},$ $x\ge0 $
imply the  following  equation for $|z|\ge 1 $
\begin{align}                                       \label{dmg1B22}
&   \bold E\left[ z^{D_{x}(\nu_{s})};
     D_{x}^{+}(\nu_{s})\le k\right]=
    \bold E\left[ z^{D_{x}(\nu_{s})};\chi_{r}^{B}(x)>\nu_{s}\right]+\notag\\
&   +\sum_{i=1}^{\infty}
    \bold E\left[e^{-s\chi_{r}^{B}(x)};T=i,\mathfrak{A}_{r} \right]
    z^{-(r+i)}
    \bold E\left[ z^{D_{x}(\nu_{s})};D_{x}^{+}(\nu_{s})\le i+B\right],
\end{align}
where $ B=k+r.$ This equation for the case of a spectrally one-sided L\'{e}vy
process was derived in \cite{TKad3}, and for the general L\'{e}vy  process
in  \cite{2Ka2}.
Let us briefly explain the equation (\ref{dmg1B22}). The increments of the
process $ D_{x}(t) $ on the interval   $[0,\nu_{s}]$  without the intersection
of the level  $k $ (the left-hand side) can be realized
either on the sample paths of the process  which do not
cross  the  negative level $-r $ (the first term of the right-hand side)
or on the sample paths which do cross  the level  $-r $
and then  the  further evolution  of the process is nothing but
its  probabilistic copy on $ [0,\nu_{s}]$ (the second term).
In  view of  (\ref{dmg1B22})  and   (\ref{dmg1B12}),
(\ref{dmg1B14})  we  find   for the  function
$
    E_{r,k}^{s}(x,z)=
   \bold E\left[ z^{D_{x}(\nu_{s})};\chi_{r}^{B}(x)>\nu_{s}\right]
$
that
\begin{align}                                     \label{dmg1B23}
     E_{r,k}^{s}(x,z)
&     = E_{k}^{+}(x,z,s)-
     \frac{Q_{k}^{s}(x)}{\bold E\,Q_{\delta+B}^{s}}
     (1-\lambda)\sum_{i\in\mathbb{N}}\lambda^{i-1} z^{-(r+i)}
     E_{i+B}^{+}(0,z,s)=\notag\\
&    = z^{k}A_{x}^{k}(s)+(1-z)\sum_{i=0}^{k-1}z^{i} A_{x}^{i}(s) +\notag\\
&    + z^{k}\frac{ Q_{k}^{s}(x)}{\hat Q_{B}^{s}(\lambda)}\,
     \frac {(1-\lambda)\check A_{0}^{B}(s,\lambda)-
     (1-z)(\lambda/z)^{B+1}\check A_{0}^{B}(s,z)}
     {1-\lambda/z},
\end{align}
where $ \check A_{0}^{B}(s,z)=\sum_{i=0}^{B}\limits z^{i} A_{0}^{i}(s), $
$
   \hat Q_{B}^{s}(\lambda)=
   \sum_{i=B+1}^{\infty}\limits\lambda ^{i}Q_{i}^{s}(0).
$
The  formula (\ref{dmg1B13}) yields
$$
     \bold P\left[\chi_{r}^{B}(x)>\nu_{s}\right]=
     A_{x}^{k}(s) +\frac{ Q_{k}^{s}(x)}{\hat Q_{B}^{s}(\lambda)}\,
     \check A_{x}^{B}(s,\lambda).
$$
It is not difficult to derive the following equality
$$
    \sum_{u=-r}^{k}z^{u}\mathfrak{E}_{r,k}^{s}(x,u)=
    \frac{1}{1-z}\left( E_{r,k}^{s}(x,z)-
    z^{k+1}\bold P\left[\chi_{r}^{B}(x)>\nu_{s}\right] \right).
$$
The  right-hand   side  of (\ref{dmg1B23}) implies that
\begin{align*}
    \sum_{u=-r}^{k}z^{u}\mathfrak{E}_{r,k}^{s}(x,u)
    = \sum_{u=0}^{k}z^{u}A_{x}^{u}(s)+
    z^{k}\frac{ Q_{k}^{s}(x)}{\hat Q_{B}^{s}(\lambda)}\,
    \sum_{i=0}^{B}(\lambda/z)^{i}
    \sum_{j=0}^{B-i} \lambda^{j} A_{0}^{j}(s).
\end{align*}
Comparing the  coefficients  of   $z^{u},$
$ u\in\overline{-r,k}, $  we find
$$
    \mathfrak{E}_{r,k}^{s}(x,u)= A_{x}^{u}(s)+
    \frac{ Q_{k}^{s}(x)}{\hat Q_{B}^{s}(\lambda)}\,\lambda^{k-u}
    \sum_{i=0}^{r+u} \lambda^{i} A_{0}^{i}(s).
$$
Since
$$
    \bold E Q_{\delta+B}^{s}=
    (1-\lambda) \lambda^{-B-1}\hat Q_{B}^{s}(\lambda), \qquad
    \sum_{i=0}^{\infty} \lambda^{i} A_{0}^{i}(s)=0,
$$
one  can see   that  the previous  equality is  the  formula (\ref{dmg1B20}).
Let us verify (\ref{dmg1B21}).
It is clear that
$$
    s\int_{0}^{\infty}e^{-st} \tilde e_{[rB],[kB]}^{tB^{2}}(x,[uB])\,dt =
    \mathfrak{E}^{s/B^{2}}_{[rB],[kB]}(x,[uB]),\qquad k\in(0,1)\quad r=1-k,
$$
where the function $ \mathfrak{E}^{s}_{r,k}(x,u) $ is determined by
(\ref{dmg1B20}).
Thus,
\begin{align}                                       \label{dmg1B24}
    \lim_{B\to\infty}
&   \int_{0}^{\infty}e^{-st}
    \tilde e_{[rB],[kB]}^{tB^{2}}(x,[uB])\,dt=
    \frac{1}{s}\lim_{B\to\infty} \mathfrak{E}^{s/B^{2}}_{[rB],[kB]}(x,[uB])
    \stackrel{\rm def}{=} e^*(s)=\notag\\
&   = \frac{1}{s}
    \left[1-\ch\left(\frac{u^{+}}{\sigma}\sqrt{2s}\right)\right]+
    \frac{1}{s}\;\frac{\sh{k}
    \sqrt{2s}/\sigma}{\sh\sqrt{2s}/\sigma}\left
    [\ch\left(\frac{r+u}{\sigma}\sqrt{2s}\right)-1\right],
\end{align}
where  $ u^{+}=\max(0,u).$  In order to  compute this limit  we
used the formulae  (\ref{dmg1B19}). Note,  that  the inversion of  the Laplace  transform
in the right-hand side  of  (\ref{dmg1B24})   this  equality
was  found  in \cite{TKad3} and resulted  into
the  following  formula $(\alpha>0)$
\begin{align*}
   \bold P
&  \left[-r\le \inf_{v\le t} w_{v},\, w_{t}\le u,\,
   \sup_{v\le t}w_{v} \le k \right]=
   \frac{1}{2\pi i}
   \int_{\alpha-i\infty}^{\alpha+i\infty}e^{st}e^*(s)\,ds=\\
&  = \frac{4}{\pi}\,\sum_{n\in\mathbb{N}}
   \frac{e^{-\frac{t}{2}(\pi n\sigma)^{2}}}{n}\,
   \sin(r\pi n)
   \sin^{2}\left(\frac{r+u}{2}\,n\pi\right), \qquad u\in[-r,k].
\end{align*}
Therefore, we  established  the weak convergence of the joint
distribution  $ \tilde e_{x}^{t}(u,B) $ as  $ B\to\infty $ to the
corresponding distribution  of the Wiener  process  and also
verified the formula (\ref{dmg1B21}).
\end{proof}

\section { Reflections from the boundary } \label{smg4}

Denote by $ D^{r}_{x}(t)=r+D_{x}(t), $ $ t\ge0 $
the process starting   from  $  r\in \mathbb{Z} $ when
$ \eta_{x}^{+}(0)=x\ge0. $
Let  $ B\in\mathbb{Z}^{+} $ and for all $ t\ge0 $ we define right-continuous
processes reflected at the boundary B as follows
\begin{align}                                    \label{dmg1B25}
&    \overline D^{B}_{r}(x,t)= D^{r}_{x}(t)
     -\max\left\{0,\sup_{[0,t]}D^{r}_{x}(\cdot)-B \right\}
     \in \overline{-\infty,B},\quad r\in \overline{-\infty,B}.
\end{align}
The  first reflection from the upper boundary $ B $ of the process
$ \overline D_{r}^{B}(x,t) $ takes place at  $\tau^{B-r}(x).$ Then
the process stays at the boundary  for  some random  time  $\eta_{l},$
where  $l=\eta^{+}_{x}(\tau^{B-r}(x)).$ At  the  instant
$t= \tau^{B-r}(x)+ \eta_{l}$ the process is reflected to a random state
$B-\delta.$ In the sequel  the  evolution of   the process
$ \overline D_{r}^{B}(x,t) $
is a probabilistic copy of  its evolution on $ [0,\tau^{B-r}(x)+ \eta_{l}).$
It is worth noticing that reflections from the boundaries reflected by
infimum (supremum) were introduced by
L\'{e}vy for a standard
Wiener process.  Applying the symmetry principle  and  the mirror   reflection
principle  L\'{e}vy  determined  the  distributions of the boundary functionals
of the reflected  standard Wiener process. We will show that these distributions
are the weak limit distributions  for the reflected process
after an appropriate  scaling of time and space.

\subsection{Passage of the  lower boundary}

We now define the boundary functionals for
process (\ref{dmg1B25}). For $ r\in\overline{0,B} $  denote
\begin{align*}
   \overline{\tau}^{B}_{r}(x) =
   \inf\{t:\overline D^{B}_{r}(x,t)<0\}
   \stackrel{\rm def}{=}\overline{\tau},\qquad
   \overline{T}^{B}_{r}(x)=
   -\overline{D}^{B}_{r}(\overline{\tau})
   \stackrel{\rm def}{=} \overline{T}, \quad r\in[0,B]
\end{align*}
the first crossing time of the lower level $0$ by the process
$ \overline D^{B}_{r}(x,t)$
and the value of the overshoot at this instant.
Note, that  these  boundary  functionals were  studied in \cite{L9}
for the reflected  L\'{e}vy processes    generated  by  infimum (supremum).
It is worth noticing  that  in this   article   the  asymptotic
expansions for the distributions of the characteristics of the process
were   determined for the reflected
L\'{e}vy processes   obeying the two-boundary  Cramer's conditions.

The  reflected  spectrally one-sided L\'{e}vy processes  generated
by the infimum (supremum) of the process  were  considered in
\cite{Avram2004}, \cite{Nguyen2005}. An interesting  application
in queueing theory  for  the  spectrally one-sided L\'{e}vy process
reflected by  its infimum   was  given  in  \cite{Bekker2008}.

\begin{theorem}    \label{tdmg1B3}

Let $ \{\overline D^{B}_{r}(x,t)\}_{t\ge0}$
be  the reflected  processes
defined by
(\ref{dmg1B25}), $B\in\mathbb{Z}^{+},$ $ r\in\overline{0,B},$
\begin{align*}
    V^{k}(x,dl,m,s)=
    \bold E\left [e^{-s\chi};L\in dl,T=m,
    \mathfrak{A}^{k}\right],\;
    V_{r}(x,m,s)=\bold E\left [e^{-s\chi};
    T=m,\mathfrak{A}_{r}\right]
\end{align*}
the  Laplace  transforms of the joint distribution of
$\{\chi^{B}_{r}(x),L,T\}$ of the process
\\$\{D_{x}(t)\}_{t\ge 0}$ \cite{2Ka13}. Then

\begin{itemize}
\item[\rm{(i)}]
if $\delta\in\mathbb{N}$ (an  arbitrarily distributed   non-negative  variable),
then the Laplace transform of
the joint distribution of $ \{\overline\tau,\overline T\} $
is such that   for $(m\in\mathbb{N})$
\begin{align}                                        \label{dmg1B26}
&   \overline v^{s}_{r}(x,m)=
    \bold E\left [e^{-s\overline{\tau}^{B}_{r}(x)};
    \overline T=m\right]
    = V_{r}(x,m,s)\notag\\
&   +\frac{a^{B-r}(x)}{1-A(0)}
    \left[\bold P[\delta=m+B]
    +\sum_{i=1}^{B} \bold P[\delta=i]
    V_{B-i}(0,m,s) \right],
\end{align}
where  $ V^{k}(x,dl,s)=\sum_{m=1}^{\infty} V^{k}(x,dl,m,s),$
\begin{align*}
   a^{k}(x)=\int_{0}^{\infty}
   V^{k}(x,dl,s)\tilde f_{l}(s),\qquad
   A(x) =\sum_{k=1}^{B}\bold P\left[\delta=k\right]a^{k}(x);
   \end{align*}
\item[\rm{(ii)}]
if  $\delta\sim ge(\lambda),$
then  the following  equalities  hold
\begin{align}                                      \label{dmg1B27}
   \overline v^{s}_{r}(x,m) =
   \frac{\tilde f_{x}(s)+
   (1-\tilde f(s))S^{s}_{B-r-1}(x)}
   {\tilde f(s)+ (1-\tilde f(s))
   \bold E\,S^{s}_{\delta+B-1}}\,
   (1-\lambda) \lambda^{m-1}, \quad r\in\overline{0,B},
\end{align}
where
$
    S^{s}_{k}(x)=\sum_{i=0}^{k} Q_{i}^{s}(x),$
    $\bold E\,S^{s}_{\delta+B-1}=
    (1-\lambda)\sum_{i=1}^{\infty}\lambda^{i-1}
    S^{s}_{i-1+B}(0);
$
the  random variable $ \overline{\tau}^{B}_{r}(x) $
is proper
$ (\bold P\left[\overline{\tau}^{B}_{r}(x)<\infty\right]=1 )$
and
\begin{align}                                  \label{dmg1B28}
   \bold E \overline{\tau}^{B}_{r}(x)=
   \bold E\eta_{x}- \bold E\eta+
   \bold E\eta\left[\bold E S_{\delta+B-1}-S_{B-r-1}(x) \right]<\infty,
\end{align}
where $ S_{k}(x)=S_{k}^{0}(x),$
$ \bold E S_{\delta+B} =\bold E S_{\delta+B}^{0};$
\item[\rm{(iii)}]
under  the conditions $(A)$  the following equality is  valid
\begin{align*}
    \lim_{B\to\infty} \bold E e^{-{s}
    \overline{\tau}^{B}_{[rB]}(x)/B^{2}}=
    \frac{\ch\left(k\sqrt{2s}/\sigma\right)}
    {\ch\left(\sqrt{2s}/\sigma\right)},  \quad r\in(0,1),\;k=1-r.
\end{align*}
\end{itemize}
\end{theorem}

\begin{proof}
Let us  verify  the formula  (\ref{dmg1B26}).
It  follows from the definition of the process
$\overline D^{B}_{r}(x,t)$ (\ref{dmg1B25}), the total probability law
and the Markov property of $\chi,$ $\eta_{n}(x)$  that the following
system of the linear integral equations holds
\begin{align*}
&  \overline v^{s}_{r}(x,m) =
   V_{r}(x,m,s)+\int_{0}^{\infty}
   V^{k}(x,dl,s) \overline v^{s}_{B}(l,m),\quad k=B-r, \\
&  \overline v^{s}_{B}(x,m)=
   \tilde f_{x}(s) \bold P[\delta=m+B]
   +\tilde f_{x}(s)\sum_{r=1}^{B}
   \bold P[\delta=r]\overline v_{B-r}^{s}(0,m).
\end{align*}
This system is similar to a system of linear equations with two unknowns and
can be solved analogously. Substituting  the  expression  for
$ \overline v^{s}_{B}(x,m) $ from the second equation into the first one,
we find that
\begin{align*}
   \overline v^{s}_{r}(x,m) = V_{r}(x,m,s)+
   a^{k}(x)\bold P[\delta=m+B]+a^{k}(x)
    \sum_{r=1}^{B} \bold P[\delta=r] \overline v_{B-r}^{s}(0,m).
\end{align*}
Letting  $x=0 $ in the latter equation after calculations yields
\begin{align*}
&    \sum_{r=1}^{B} \bold P[\delta=r] \overline v_{B-r}^{s}(0,m)=
     - \bold P[\delta=m+B] \\
&    + \left[\bold P[\delta=m+B]+ \sum_{k=1}^{B}
    \bold P[\delta=k]  V_{B-k}(0,du,s)\right](1-A(0))^{-1}.
\end{align*}
Inserting the right-hand side of this quality in the previous one,
we get (\ref{dmg1B26}).
In case  when $\delta\sim ge(\lambda) $ the  formula (\ref{dmg1B26})
takes  a more simple  form. The  first formula of (\ref{dmg1B12})
and (\ref{dmg1B26}) imply that  $ \overline T_{r}^{B}(x) \sim ge(\lambda)$
for any  $r\in\overline{0,B}.$
Summing over $m\in\mathbb{N}$   both  sides of  (\ref{dmg1B26}),  we find
for the function
$
\overline v^{s}_{r}(x)= \bold E e^{-s\overline{\tau}^{B}_{r}(x)}
$
that
\begin{align}                                    \label{dmg1B29}
    \overline v^{s}_{r}(x)= V_{r}(x,s)
    +\frac{a^{B-r}(x)}{1-A(0)}
    \left[\lambda^{B} + (1-\lambda)\sum_{i=1}^{B}\lambda^{i-1}
    V_{B-i}(0,s) \right].
\end{align}
Now we calculate
$ a^{k}(x),$  $ A(0) $ in case  when  $ \delta\sim ge(\lambda).$
Employing  the formulae (\ref{dmg1B9}), (\ref{dmg1B12}) and
performing the necessary calculations  we find that
\begin{align*}
&      a^{k}(x)=
      \tilde f_{x}(s)+ (1-\tilde f(s))S^{s}_{k-1}(x)
      -  \frac{Q_{k}^{s}(x)}{\bold E\,Q_{\delta+B}^{s}}
      \left[\tilde f(s)+
      (1-\tilde f(s))\bold ES_{\delta+B-1}^{s}\right],\\
&      1-A(0)= \frac{1-\lambda}{\lambda \bold E\,Q_{\delta+B}^{s}}
      \left[1-Q_{0}^{s}(0) \right]
      \left[\tilde f(s)+
       (1-\tilde f(s))\bold ES_{\delta+B-1}^{s}\right],\\
&        \lambda^{B} + (1-\lambda)\sum_{i=1}^{B}\lambda^{i-1}
      V_{B-i}(0,s)= \frac{1-\lambda}{\lambda \bold E\,Q_{\delta+B}^{s}}
      \left[1-Q_{0}^{s}(0) \right].
\end{align*}
Substituting the right-hand sides of these equalities into (\ref{dmg1B29})
and  taking  into account  that $ \overline T_{r}^{B}(x) \sim ge(\lambda),$
we  derive  (\ref{dmg1B27}).

The limiting  formulae  (\ref{dmg1B19})  were
obtained  in \cite{3Ka1}. Similarly (see \cite{3Ka1}) we derive for all
$k>0$ that
\begin{align}                                       \label{dmg1B30}
   \lim_{B\to\infty}
    {B^{-2}}\,S _{[kB]}^{s/B^{2}}(x)=
    \frac{1}{s\bold E\eta}\left(\ch\left
    (k\sqrt{2s}/\sigma\right)-1\right)=
    \lim_{B\to\infty}
    {B^{-2}}\,\bold E S_{[kB]+\delta}^{s/B^{2}}.
\end{align}
Letting  $ B\to\infty,$ we have
$ \tilde f_{x}(s/B^{2})=1-\bold E\eta_{x}s/B^{2}+o(s/B^{2}),$
which implies   for $ r\in(0,1)$ that
\begin{align*}
    \lim_{B\to\infty} \bold E e^{-{s}
    \overline{\tau}^{B}_{[rB]}(x)/B^{2}}=
              \frac{1+(\ch\left(k\sqrt{2s}/\sigma)-1)\right)}
    {1+(\ch\left(\sqrt{2s}/\sigma)-1)\right)}=
    \frac{\ch\left(k\sqrt{2s}/\sigma\right)}
    {\ch\left(\sqrt{2s}/\sigma\right)}, \quad k=1-r.
\end{align*}
The equality  (\ref{dmg1B28})   follows  from  the  following
$
\bold E\overline{\tau}^{B}_{r}(x)=
-\left.\frac{d}{ds}\overline v^{s}_{r}(x)\right|_{s=0}.
$
\end{proof}

\subsection{Increments  of the process  reflected in its supremum}

Define
$
   \overline D^{k}_{0}(x,t)= D_{x}(t)
   -\max\left\{0,\sup_{[0,t]}\limits D_{x}(\cdot)-k \right\}
   \in \overline{-\infty,k},
$
the  process reflected  from the upper  boundary  $ k\in\mathbb{Z}^{+}$
generated by its  supremum.
\begin{theorem}    \label{tdmg1B4}

Let  $ \{\overline D^{k}_{0}(x,t)\}_{t\ge0}$  be the   process
reflected from the upper boundary and
$
   \overline p_{k}^{s}(x,u)=
   \bold P\left[\overline D^{k}_{0}(x,\nu_{s})\le u \right], $
   $ u\in\overline{-\infty,k}
$
be the distribution of its increments on the exponential  interval
$ [0,\nu_{s}].$ Then
\begin{itemize}
\item[\rm{(i)}]
for all $ k\in\mathbb{Z}^{+}$ $\overline p_{k}^{s}(x,k)=1;$ and  for
$u\in\overline{-\infty,k-1}, $ $ x\ge0 $
the  following  equality  holds
\begin{align}                                       \label{dmg1B31}
    \overline p_{k}^{s}(x,u)=
    A_{x}^{u}(s)+c(s)^{k-u-1}F(s)
    \left( \frac{\tilde f_{x}(s)}{1-\tilde f(s)}
    +S_{k-1}^{s}(x)\right),
\end{align}
where   $ A_{x}^{u}(s)=0,$ for $u<0,$
$ F(s)=s(1-c(s))/(1-\lambda)(s-k(c(s))); $
\item[\rm{(ii)}]
under the  conditions $(A)$
for $ k>0,$ $ u\le k$  the  following  relation is  valid
\begin{align}                                      \label{dmg1B32}
    \lim_{B\to\infty}
    \bold P \left[\overline D^{[kB]}_{0}(x,tB^{2})\le [uB] \right]=
    1-\frac{1}{\sigma\sqrt{2\pi t}}
    \int_{u}^{2k-u}e^{-v^{2}/2\sigma^{2}t}dv;
\end{align}
\item[\rm{(iii)}]
if $ \rho>1,$
then the   ergodic  distribution
$
    p_{k}(u)=\lim_{t\to\infty}\limits
    \bold P \left[\overline D^{k}_{0}(x,t)\le u \right]
$
exists,
$$
     p_{k}(u)=
     \frac{\bold E\varkappa}{\rho}
     \frac{1-c}{1-\bold Ec^{\varkappa}}\,c^{k-u-1},
     \quad u\in\overline{-\infty, k-1}, \qquad
     c=\lim_{s\to 0}\limits c(s)\in(\lambda,1).
$$
\end{itemize}
\end{theorem}

\begin{proof}
Define the generating function distribution  of the process
$$
   \overline P_{k}^{s}(x,z)=\bold E z^{\overline D^{k}_{0}(x,\nu_{s})}=
   \sum_{-\infty}^{k}z^{i}
   \bold P\left[\overline D^{k}_{0}(x,\nu_{s})=i\right],
   \quad |z|\ge 1,\;k\in\mathbb{Z}^{+}.
$$
It is  obvious that
$ \overline P_{k}^{s}(x,1) =
\bold P\left[\overline D^{k}_{0}(x,\nu_{s})\le k\right]=1.$
In accordance with the total probability law and  the definition of
the process $\overline D^{k}_{0}(x,t) $
we can write
\begin{align}                                      \label{dmg1B33}
     \overline P_{k}^{s}(x,z)
&    = E_{k}^{+}(x,z,s)+ z^{k}\int_{0}^{\infty}
     f_{x}^{k}(dl)(1-\tilde f_{l}(s)) \notag\\
&    + z^{k}\int_{0}^{\infty}f_{x}^{k}(dl)\tilde f_{l}(s)
     (1-\lambda)\sum_{i=1}^{\infty}\lambda^{i-1} z^{-i}
     \overline P_{i}^{s}(0,z),
\end{align}
where  the  function
$
    E_{k}^{+}(x,z,s)=\bold E \left[e^{-zD_{x}(\nu_{s})};
    \tau^{k}(x)>\nu_{s}\right]
$
is given  by  (\ref{dmg1B14}), and the  function
$  f^{k}_{x}(dl)=\bold E\left[e^{-s\tau^{k}(x)};\eta^{k}(x)\in dl \right] $
is determined  by (\ref{dmg1B9}). This  equation means  the  following.
The sample  paths  on which the  increments of the process  $ D_{x}(t)$
occur  can be decomposed  into  three  types:
1)  the  sample  paths which  do not intersect the   upper boundary  $k$
(the first  term of the right-hand side);
2)  the sample paths which do intersect
the  upper boundary and   stay there (the  second term);
3) the sample paths   cross the upper   boundary and
 then they are reflected (the third term).
After some   calculation  which we  skip, one can see  that the formula
(\ref{dmg1B9}) implies  that
\begin{align*}
   \int_{0}^{\infty}f_{x}^{k}(dl)\tilde f_{l}(s)=
   \tilde f_{x}(s)+(1-\tilde f(s))S_{k}^{s}(x)-
   \frac{1-\tilde f(s)}{1-c(s)}\,Q_{k}^{s}(x).
\end{align*}
Letting   $ x=0 $ in (\ref{dmg1B33}) and taking into account the latter
equality, we derive
\begin{align*}
     (1-\lambda)\sum_{i=1}^{\infty}\lambda^{i-1} z^{-i}
     \overline P_{i}^{s}(0,z)=
   \frac{F(s)}{1-\tilde f(s)} \,
   \frac{1/z-1}{1-c(s)/z},\qquad |z|\ge 1.
\end{align*}
Substituting the right-hand of this equality and the expression
(\ref{dmg1B14}) for $ E_{k}^{+}(x,z,s) $ into (\ref{dmg1B33}), we find that
$( |z|\ge 1)$
\begin{align*}
    \overline P_{k}^{s}(x,z)
    =z^{k}+ (1-z)\sum_{i=0}^{k-1}z^{i}A_{x}^{i}(s)
    +z^{k}F(s)\frac{1/z-1}{1-c(s)/z}\,
     \left( \frac{\tilde f_{x}(s)}{1-\tilde f(s)}+ S_{k-1}^{s}(x)\right).
\end{align*}
Comparing  the coefficients  of  $z^{i},$ $ i\in\{k,k-1,\dots\}, $
we get
\begin{align*}
&     \bold P\left[\overline D^{k}_{0}(x,\nu_{s})=k\right]
    =  1- A_{x}^{k-1}(s)- F(s)
     \left(\frac{\tilde f_{x}(s)}{1-\tilde f(s)}+S_{k-1}^{s}(x)\right),\\
&     \bold P\left[\overline D^{k}_{0}(x,\nu_{s})=i\right]
    =  A_{x}^{i}(s)- A_{x}^{i-1}(s)+ \\
&    + F(s) c(s)^{k-i-1}(1-c(s))
     \left( \frac{\tilde f_{x}(s)}{1-\tilde f(s)}+ S_{k-1}^{s}(x)\right),
     \qquad i<k.
\end{align*}
One can see that the second formula implies the equality (\ref{dmg1B31}) of the
theorem. Let us verify (\ref{dmg1B32}). It follows from the first formula of
(\ref{dmg1B19}) that
\begin{align*}
    F(s/B^{2})=\frac{s}{B^{2}}\,\bold E\eta+o(B^{-2}),\qquad
    \lim_{B\to\infty}c(s/B^{2})^{[B(k-u)]-1}=e^{-(k-u)\sqrt{2s}/\sigma}.
\end{align*}
Denote
$
   \tilde p_{k}^{t}(x,u,B)=
   \bold P \left[\overline D^{[kB]}_{0}(x,tB^{2})\le [uB] \right],
$
$ k>0.$  Then (\ref{dmg1B19}), (\ref{dmg1B30})
\begin{align*}
    \lim_{B\to\infty}
&   \int_{0}^{\infty}e^{-st}\tilde p_{k}^{t}(x,u,B)\,dt=
     \frac{1}{s}\lim_{B\to\infty} \overline p_{[kB]}^{s/B^{2}}(x,[uB])= \\
&    =\frac{1}{s}\left(1-\ch\left(u^{+}\sqrt{2s}/\sigma\right) +
     e^{-(k-u)\sqrt{2s}/\sigma}\ch\left(k\sqrt{2s}/\sigma\right)\right)=\\
&    =\frac{1}{s}\, \bold I_{\{u<0\}}\left(e^{u\sqrt{2s}/\sigma}/2
    +e^{-(2k-u)\sqrt{2s}/\sigma}/2\right)+\\
&   + \frac{1}{s}\, \bold I_{\{u\in[0,k]\}}
    \left(1-e^{-u\sqrt{2s}/\sigma}/2+e^{-(2k-u)\sqrt{2s}/\sigma}/2\right),
    \quad u\le k,
\end{align*}
where $u^{+}=\max\{0,u\}.$  Employing the formula
$
    \frac{1}{s}\,e^{-a\sqrt{2s}/\sigma}=
    2\int_{0}^{\infty}e^{-st} \bold P\left[w_{t}\ge a\right]dt,
$
to invert  the Laplace transform, we  derive the limiting  equality of the
theorem.

For $ \rho>1 $  the  mathematical  expectation of
$ \tau^{k}(x) $ is finite. It follows from  (\ref{dmg1B10})  that
$$
     \bold E \tau^{k}(x) = \frac{Q_{k}(x)}{(1-\lambda) k(c)}+
        \sum_{i=0}^{k}\rho_{i}\,
   \left[1-\frac{Q_{k-i}(x)}{1-\lambda}\right]<\infty,
$$
where
$
  \rho_{i}=\lim_{s\to 0}\limits s^{-1}\tilde\rho_{i}(s)=
  \int_{0}^{\infty}\bold P\left[\pi(t)=i\right]dt<\infty.
$
Moreover,   the process  $ \overline D^{k}_{0}(x,t) $  is of  regenerative type
 \cite{ShK1}.  The instants if the   passages of the  upper  boundary are the
regeneration times. Hence, \cite{ShK1}  there  exists  the
ergodic distribution  of the process
$
    p_{k}(u)=\lim_{t\to\infty}\limits
    \bold P \left[\overline D^{k}_{0}(x,t)\le u \right].
$
To determine  this distribution,  it suffices  to  apply  to   (\ref{dmg1B31})
the Tauberian  theorem:
$
    p_{k}(u)=\lim_{s\to 0}\limits \overline p_{k}^{s}(x,u).
$
\end{proof}

Let  $ \{\overline D^{k}_{0}(x,t)\}_{t\ge0}$  be  the process
reflected from the  upper  boundary.   Define   for  $ r,k \in\mathbb{Z}^{+}$
\begin{align*}
   \overline{\tau}_{r,k}(x) =
   \inf\{t:\overline D^{k}_{0}(x,t)<-r\}
   \stackrel{\rm def}{=}\overline{\tau},\qquad
   \overline{T}_{r,k}(x)=
   -\overline{D}^{k}_{0}(x,\overline{\tau})-r
   \stackrel{\rm def}{=} \overline{T},
\end{align*}
the  first exit time    from the interval $ [-r,k] $ by the  process
$ \overline D^{k}_{0}(x,t) $ and the  value  of the overshoot  through
the  lower boundary  $-r.$
Since  $ X_{t} $  is homogeneous  with respect to the  first  component,
then    $ \{\overline{\tau}_{r,k}(x),\overline{T}_{r,k}(x)\} $
are  identically distributed   as
$ \{\overline{\tau}_{r}^{B}(x),\overline{T}_{r}^{B}(x)\}, $
$ B=k+r $ and  their joint  distribution is  determined  by  (\ref{dmg1B27}).

\begin{theorem}    \label{tdmg1B5}

Let $ \{\overline D^{k}_{0}(x,t)\}_{t\ge0}$ be the process reflected from the
upper boundary,
$
   \overline p_{r,k}^{s}(x,u)=
   \bold P\left[\overline D^{k}_{0}(x,\nu_{s})\le u;\;
   \overline{\tau}_{r,k}(x)>\nu_{s} \right], $
   $ u\in\overline{-r,k}
$
be the distribution of the increments  of the process  on the interval
$ [0,\nu_{s}]$ on the  event $ \{\overline{\tau}_{r,k}(x)>\nu_{s}\}.$
\begin{itemize}
\item[\rm{(i)}]
the distribution of the increments is such that
for all $ r,k\in\mathbb{Z}^{+},$
\begin{align}                                      \label{dmg1B34}
&       \overline p_{r,k}^{s}(x,k)=1-
       \frac{\tilde f_{x}(s)+
       (1-\tilde f(s))S^{s}_{k-1}(x)}
       {\tilde f(s)+ (1-\tilde f(s))
       \bold E\,S^{s}_{\delta+B-1}},\quad B=k+r,  \\
&      \overline p_{r,k}^{s}(x,u)=A_{x}^{u}(s)-
       \frac{\tilde f_{x}(s)+
       (1-\tilde f(s))S^{s}_{k-1}(x)}
       {\tilde f(s)+ (1-\tilde f(s))
       \bold E\,S^{s}_{\delta+B-1}}\,
       \bold E A_{0}^{\delta+u+r}(s),\; u\in\overline{-r,k-1};\notag
\end{align}
\item[\rm{(ii)}]
under the conditions  $(A)$ the following
limiting equality holds
\begin{align}                                       \label{dmg1B35}
     \lim_{B\to\infty}
&    \bold P \left[\overline D^{[kB]}_{0}(x,tB^{2})\le [uB];\;
     \overline{\tau}_{[rB],[kB]}(x)>tB^{2} \right]
     \stackrel{\rm def}{=} p(t)=\\
=&    \frac{4}{\pi}\,
     \sum_{n\in\mathbb{Z}^{+}}
     \frac{e^{-\frac{t}{2}(\pi(n+\frac{1}{2})\sigma)^{2}}}
     {n+\frac{1}{2}}\,
     \sin \left(r\left( n+\frac{1}{2}\right)\pi\right)
     \sin^{2}\left(\frac{r+u}{2}\left(n+\frac{1}{2}\right)\pi\right), \notag
\end{align}
where $ r\in(0,1), $ $k=1-r, $ $u\in[-r,k].$
\end{itemize}
\end{theorem}

\begin{proof}
In accordance    with the   total probability  law, homogeneity  of
the process  $ X_{t} $ with  respect   to  the  first   component,
Markov property  of $ \overline{\tau}_{r,k}(x) $
and the  properties of the   exponential variable $ \nu_{s} $
we can  write
\begin{align*}
    \overline p_{k}^{s}(x,u)=\overline p_{r,k}^{s}(x,u)+
    \overline v^{s}_{r}(x)(1-\lambda)\sum_{i=1}^{\infty}\lambda^{i-1}
    \overline p_{i+B}^{s}(0,u+r+i), \quad u\in\overline{-\infty,k},
\end{align*}
where the function
$
    \overline p_{k}^{s}(x,u)=
    \bold P\left[\overline D^{k}_{0}(x,\nu_{s})\le u \right]
$
is determined   in Theorem  \ref{tdmg1B4}.   This   equation  means   that
increments  of the process  $ \overline D^{k}_{0}(x,\nu_{s}) $ are  realized
either on the  sample paths   which do not  exit the interval
$ [-r,k],$  or on the sample  paths  which   do  exit  the interval
and  the further evaluation of the process is its probabilistic  replica  on $ [0,\nu_{s}].$
Substituting   the  expression for the  function $ \overline p_{k}^{s}(x,u) $
into (\ref{dmg1B31})  after  necessary calculations  we  derive
(\ref{dmg1B34}).

For $ r\in(0,1),$ $k=1-r,$ $ u\in[-r,k] $ denote
$$
    \overline p_{r,k}^{t}(x,u,B)=
    \bold P \left[\overline D^{[kB]}_{0}(x,tB^{2})\le [uB];\;
    \overline{\tau}_{[rB],[kB]}(x)>tB^{2} \right].
$$
Employing the third  formula of  (\ref{dmg1B19}),
the limiting  equality of Theorem \ref{tdmg1B3},
we find 
\begin{align}                                       \label{dmg1B36}
&    \frac{1}{s}\lim_{B\to\infty}
     \overline p_{[kB]}^{s/B^{2}}(x,[uB])=
    \lim_{B\to\infty}\int_{0}^{\infty}e^{-st}
    \overline p_{r,k}^{t}(x,u,B)\,dt=\\
&    =\frac{1-\ch(u^{+}\sqrt{2s}/\sigma)}{s}+
     \frac{1}{s}\,\frac{\ch(k\sqrt{2s}/\sigma)}{\ch(\sqrt{2s}/\sigma)}
     \left(\ch((u+r)\sqrt{2s}/\sigma)-1\right)
     \stackrel{\rm def}{=} p^{*}(s), \notag
\end{align}
where   $ u^{+}=\max\{0,u\}.$  When $ u\in[-r,0] $ we  derive from this formula that
\begin{align*}
    p^{*}(s)= \frac{2}{s}\,
     \frac{\ch(k\sqrt{2s}/\sigma)}{\ch(\sqrt{2s}/\sigma)}\,
     \sh^{2}\left(\frac{r+u}{2}\sqrt{2s}/\sigma\right),\qquad  u\in[-r,0].
\end{align*}
It is  clear that $ s=0 $ is not  a singular point (pole or  point of branching) of the function
$ p^{*}(s). $  In the  semi-plane $ \Re(s)<0 $ this function has
simple poles in
$$
     s_{n}=-\frac{1}{2}\,\sigma^{2}\pi^{2}\left(n+\frac{1}{2}\right)^{2},
     \qquad n\in\mathbb{Z}^{+},
$$
and it is analytic in the whole plane apart from these points. Hence, for
$ \alpha>0 $
\begin{align*}
   p(t)=\frac{1}{2\pi i}
   \int_{\alpha-i\infty}^{\alpha+i\infty}e^{st}p^*(s)\,ds=
   \sum_{n\in\mathbb{Z}^{+}}{\rm Res}_{s=s_{n}} p^*(s).
\end{align*}
Calculating  the   residues  of the function  $ p^*(s) $ in  $ s_{n}, $  we   obtain
the right-hand  side of the formula (\ref{dmg1B35}) for $ u\in[-r,0]. $ On can see
that  the  first  term in the  right-hand side of  (\ref{dmg1B36}) is analytic
in the whole  plane for $ u\in(0,k]. $  Applying the inversion   formula,
we   find  that  the contour integral of this  term  is equal  to  zero.
The second  term of (\ref{dmg1B36}) is the same   also for $ u\in[-r,0].$
Thus,  the  formula (\ref{dmg1B35}) holds for  $ u\in[-r,k]. $
\end{proof}

\section{Applications for$ {\rm M^{\varkappa}|G^{\delta}|1|B} $ system }

Let  $B\in\mathbb{Z}^{+},$  $r\in\overline{0,B+1},$ $x\ge0.$
Introduce  the two-component  Markov process
$$
   Y_{r,x}(t)= \{d_{r,x}(t),\eta_{r,x}(t)\}
   \in\overline{0,B+1}\times \mathbb{R}_{+},\quad  Y_{r,x}(0)=(r,x)
$$
by means of the following
recurrent equations
\begin{align}                                    \label{dmg1B37}
   Y_{r,x}(t)= \left\{
   \begin{array}{l}
   \left(\overline D^{B+1}_{r}(x,t),\eta_{x}^{+}(t)\right),
   \quad  0\le t< \tilde\tau_{r}^{B+1}(x), \\
   Y_{0,0}(t-\tilde\tau_{r}^{B+1}(x)),
   \qquad t\ge\tilde\tau_{r}^{B+1}(x),
   \end{array}
   \right. ,\quad r\in\overline{1,B+1},
\end{align}
\begin{align}                                    \label{dmg1B38}
   Y_{0,0}(t)= \left\{
   \begin{array}{l}
   (0,0),
   \quad  0\le t< \tilde\mu\sim\exp(\mu), \\
   Y_{r,0}(t-\tilde\mu):\bold P[\varkappa=r],\quad r=\overline{1,B},
   \quad t\ge\tilde\mu, \\
   Y_{B+1,0}(t-\tilde\mu):\bold P[\varkappa\ge B+1],
   \quad t\ge\tilde\mu,
   \end{array}
   \right.
\end{align}
where
$ \tilde\tau_{r}^{B+1}(x)=\inf\{t:\overline D^{B+1}_{r}(x,t)<1\},$
$ r=\overline{1,B+1}.$

\begin{remark}   \label{rdmg1B1}

Since  the process  $ X_{t} $ is homogeneous  with respect to the first
component, then  the random  variable  $ \tilde\tau_{r}^{B+1}(x)$ is identically
distributed as $ \overline\tau_{r-1}^{B}(x) $ (\ref{dmg1B27}) and, hence,
$$
    \tilde v_{r}^{s}(x)=
    \bold E\left[e^{-s\tilde\tau_{r}^{B+1}(x)};
    \tilde\tau_{r}^{B+1}(x)<\infty\right]= \overline v_{r-1}^{s}(x),
    \quad r\in\overline{1,B+1}.
$$
\end{remark}

The  process $  Y_{r,x}(t)_{\{t\ge0\}} $ serves as a mathematical model of
the  functioning of ${\rm M^{\varkappa}|G^{\delta}|1|B},$ $(\delta\sim ge(\lambda))$
system, which has  the  following properties
\begin{itemize}

\item[\rm{(i)}]
The customers  arrive in groups (batch arrivals)
according to the Poisson  process  with  intensity  $\mu>0.$
The number of  the  customers   in every group is  represented  by the  random variable
$ \varkappa \in\mathbb{N}. $

\item[\rm{(ii)}]
The system has a finite  waiting  room (buffer) whose
size is equal to  $ B+1<\infty.$
Suppose that  upon the arrival of a new claim
of size $ \varkappa $ it finds $ r\in\overline{0,B+1}$ occupied space
in the waiting  room.
Then $ \min\{k,\varkappa\} $ joins the queue, and  loss of size
$ \max\{0,\varkappa-k\}$ occurs, where  $k=B+1-r$ is the size of empty space
in the waiting room;

\item[\rm{(iii)}]
The duration of service completion is arbitrary distributed as
$\eta>0. $ Suppose, that  at time $t $ the service cycle is  accomplished.
Then  the  occupied  space in the buffer is reduced  by $ \min\{r,\delta\}, $
where  $ r\in\overline{1,B+1} $ is the  value of occupied  space  in the
waiting  room at time $t-0.$  If at the instant of
the service completion $r- \min\{r,\delta\}>0, $
then a new service cycle starts.  If at the instant of the service completion
$r- \min\{r,\delta\}=0, $ then the new service cycle starts  upon arrival  of a new
 claim (after exponential time  with parameter $\mu>0 $).
\end{itemize}

For  all  $t\ge0 $ the  event $ \{Y_{r,x}(t)=(i,y)\},$ $i\in\overline{1,B+1},$
$ y\ge0 $ means that    at time  $t$  there are  $i$  occupied places  in
the waiting room,  and $y$  stands for  time  elapsed since the  beginning
of the service cycle. Here $ (r,x) $  is an initial state of the system.

The event  $ \{Y_{r,x}(t)=(0,0)\}$  means  that at time  $t$
the waiting room is empty and the server is idle.
The  system stays in the  $ (0,0) $  state for an exponential period of time
(with parameter  $ \mu.)$

Therefore,  $ d_{r,x}(t) $  is the number of the customers in the waiting room
at time $t.$ If $ d_{r,x}(t)>0,$
then  $ \eta_{r,x}(t) $  is the time elapsed  since the last  start of the
service cycle up to time t. If  $ d_{r,x}(t)=0,$ then
$\bold P [\eta_{r,x}(t)=0]=1. $

\subsection {Busy period of the system}

Assume that  at   time $ t_{0}=0 $  system   is in the state $(r,x),$
where $ r\in\overline{1,B+1}$ is the number of the customers in   the waiting   room, and
 $x\ge 0 $ is    the duration of the  current   service  cycle.
 Introduce the  random variable
$$
   b_{r}(x)=\inf\{t:d_{r,x}(t)=0 \}
$$
i.e.  the  instant  at  which  the system  for the first  time becomes  empty.
Thus, the interval
$ [0,b_{r}(x)] $  is a busy    period  of  $ (r,x)$  type.

\begin{corollary}   \label{cdmg1B1}

Let
$ b_{r}^{s}(x)=\bold E\left[e^{-s b_{r}(x)}; b_{r}(x)<\infty\right] $
be the Laplace  transform of the busy  period of type $(r,x).$
Then  the following relation holds
\begin{align}                                      \label{dmg1B39}
   b_{r}^{s}(x) =
   \frac{\tilde f_{x}(s)+
   (1-\tilde f(s))S^{s}_{B-r}(x)}
   {\tilde f(s)+ (1-\tilde f(s))
   \bold E\,S^{s}_{\delta+B-1}},\quad x\ge0,
    \quad r\in\overline{1,B+1},
\end{align}
where
$$
    S^{s}_{k}(x)=\sum_{i=0}^{k} Q_{i}^{s}(x),\quad
    \bold E\,S^{s}_{\delta+B-1}=
    (1-\lambda)\sum_{i=1}^{\infty}\lambda^{i-1}
    S^{s}_{i-1+B}(0).
$$
Observe, that the  random variable $ b_{r}(x) $
is proper  $ (\bold P\left[b_{r}(x)<\infty\right]=1 )$
and
\begin{align}                                  \label{dmg1B40}
   \bold E b_{r}(x)=
   \bold E\eta_{x}- \bold E\eta+
   \bold E\eta\left[\bold E S_{\delta+B-1}-S_{B-r}(x) \right]<\infty,
\end{align}
where
$
    S_{k}(x)=S_{k}^{0}(x),$
    $\bold E S_{\delta+B} =\bold E S_{\delta+B}^{0}.
$
\end{corollary}
These formulae follow straightforwardly from Theorem  \ref{tdmg1B3}
and Remark \ref{rdmg1B1}.

\subsection {Time of the first loss of a customer}

Suppose that the system starts functioning from the state $ (r,x) $ and denote
by $ l_{r}(x) $ the time of the first loss of the customer (a group of customers).

\begin{corollary}   \label{cdmg1B2}

Let
$ l_{r}^{s}(x)=\bold E\left[e^{-s l_{r}(x)}; l_{r}(x)<\infty\right] $
be the Laplace transform of  $ l_{r}(x).$
Then  the following relation is valid
\begin{align}                                      \label{dmg1B41}
&    l_{0}^{s}(0)=1-\bold E A_{0}^{\delta+B}(s)+
    \bold EQ^{s}_{\delta+B}
    \frac{\bold E A_{0}^{\delta+B}(s)-
    \frac{\mu}{s+\mu}\tilde A(s)-\frac{s}{s+\mu}}
    {\bold EQ^{s}_{\delta+B}-\frac{\mu}{s+\mu}\,\tilde Q(s)},\notag\\
&    l_{r}^{s}(x)=1- A_{x}^{k}(s)+
    Q^{s}_{k}(x)
    \frac{\bold E A_{0}^{\delta+B}(s)-
    \frac{\mu}{s+\mu}\tilde A(s)-\frac{s}{s+\mu}}
    {\bold EQ^{s}_{\delta+B}-\frac{\mu}{s+\mu}\,\tilde Q(s)},
\end{align}
where  $ r\in\overline{1,B+1},$ $k=B+1-r,$
$$
    \tilde Q(s)=
    \sum_{i=1}^{B+1}\bold P[\varkappa=i] Q_{B+1-i}^{s},\qquad
    \tilde A(s)=
    \sum_{i=1}^{B+1}\bold P[\varkappa=i] A_{0}^{B+1-i}(s).
$$
Note, that the random variables $ l_{0}(0), $ $ l_{r}(x) $
are proper,  and they have finite mathematical expectations.
\end{corollary}

\begin{proof}
Let  $  r\in\overline{1,B+1}, $  $x\ge 0.$   Denote  by
$ \tilde\chi_{r}^{B+1}(x)=\inf\{t: r+ D_{x}(t)\notin[1,B+1] \} $
the first exit time  from  the interval $ [1,B+1]$ by the process $ r+ D_{x}(t). $
Since  the process $ X_{t} $ is homogeneous with respect to the first component,
then the random variable  $\tilde \chi_{r}^{B+1}(x) $
is identically distributed as  $ \chi_{r-1}^{B}(x) $
and its Laplace transforms are determined by the formulae of Theorem \ref{tdmg1B1}.
In accordance with the definition of the process $ Y_{r,x}(t) $
we can write the following system of the  equations for the functions
$ l_{r}^{s}(x),$ $ l_{0}^{s}(0) $
\begin{align}                                   \label{dmg1B42}
&     l_{r}^{s}(x)=V^{B+1-r}(x,s)+
      V_{r-1}(x,s)\,l_{0}^{s}(0),
      \qquad  r\in\overline{1,B+1},\;  x\ge 0, \notag \\
&     l_{0}^{s}(0)=\frac{\mu}{s+\mu}\hat a_{B+1}+
      \frac{\mu}{s+\mu}\sum_{i=1}^{B+1}
      a_{i} l_{i}^{s}(0),
\end{align}
where $a_{i}=\bold P[\varkappa=i],$ $\hat a_{i}=\bold P[\varkappa>i],$
and (\ref{dmg1B13})
\begin{align}                                       \label{dmg1B43}
&  V_{r-1}(x,s)=
   \frac{Q_{B+1-r}^{s}(x)}
   {\bold E\,Q_{\delta+B}^{s}},\\
&  V^{B+1-r}(x,s)=
   1- A^{B+1-r}_{x}(s)-
   \frac{Q_{B+1-r}^{s}(x)} {\bold E\,Q_{\delta+B}^{s}}
   \left(1- \bold E\,A^{\delta+B}_{0}(s) \right),\notag
\end{align}
Substituting the right-hand side of the first equation of (\ref{dmg1B42}) for
$x=0$ into the second one,  we get
\begin{align*}
      l_{0}^{s}(0)=\frac{\mu}{s+\mu}\,\hat a_{B+1}+
      \frac{\mu}{s+\mu}\sum_{i=1}^{B+1}
      a_{i}V^{B+1-i}(0,s)+
      \frac{\mu}{s+\mu}\sum_{i=1}^{B+1}
      a_{i}V_{i-1}(0,s) l_{0}^{s}(0).
\end{align*}
The latter equation yields
\begin{align*}
   l_{0}^{s}(0)= \frac{\mu}{s+\mu}
   \left(\hat a_{B+1}+
   \sum_{i=1}^{B+1}\limits a_{i}V^{B+1-i}(0,s)\right)
   \left(1-\frac{\mu}{s+\mu}
   \sum_{i=1}^{B+1}\limits a_{i}V_{i-1}(0,s)\right)^{-1}.
\end{align*}
Taking into account the first formula of (\ref{dmg1B42})
and (\ref{dmg1B43}),
we derive  the equalities (\ref{dmg1B41}) of the  corollary.
\end{proof}

\subsection{Number of the  customers in the system }

Let  $\nu_{s}\sim\exp(s) $ be the exponential  random variable
with parameter $ s>0. $  Introduce   the  transient  probabilities
of the process   $ d_{r,x}(t)_{\{t\ge0\}}: $
\begin{align*}
&   q_{r,x}^{s}(u) =\bold P[d_{r,x}(\nu_{s})\le u],\quad
    q_{0,0}^{s}(u) =\bold P[d_{0,0}(\nu_{s})\le u],\quad
    r,u\in\overline{1,B+1},\\
&   q_{r,x}^{s}(0)=\bold P[d_{r,x}(\nu_{s})=0],\quad
    q_{0,0}^{s}(0)=\bold P[d_{0,0}(\nu_{s})=0].
\end{align*}
Denote
$
\tilde b(s)=  \hat a_{B}\, b_{B+1}^{s}(0)+
    \sum_{i=1}^{B} a_{i}\, b_{i}^{s}(0).
$

\begin{theorem}  \label{tdmg1B6}

The  distribution of the number of the  customers  at  time $ \nu_{s} $
is  such that
\begin{align*}
&     q_{0,0}^{s}(u) = \bold E A_{0}^{\delta+u-1}(s)
     +\frac{C_{u}(s,\lambda)}{s+\mu-\mu\tilde b(s)}, \qquad
     q_{0,0}^{s}(B+1) =1-
     \frac{s}{s+\mu-\mu\tilde b(s)},\notag\\
&    q_{r,x}^{s}(u)=  A_{x}^{u-r}(s) +b_{r}^{s}(x)
     \frac{C_{u}(s,\lambda)} {s+\mu-\mu\,\tilde b(s)},\qquad
     q_{r,x}^{s}(B+1)= 1- \frac{s b_{r}^{s}(x)}{s+\mu-\mu\tilde b(s)},\\
&     q_{0,0}^{s}(0) =
     \frac{s}{s+\mu-\mu\tilde b(s)},\qquad
     q_{r,x}^{s}(0)=  \frac{s\, b_{r}^{s}(x)}{s+\mu-\mu\,\tilde b(s)},
\end{align*}
where
$$
    C_{u}(s,\lambda)=\frac{sQ_{u}^{s}}{1-\lambda}-s +
    \lambda(s+\mu)\left( A_{0}^{u}(s)- \bold E A_{0}^{\delta+u}(s)\right).
$$
\end{theorem}

\begin{corollary} \label{cdmg1B3}

Let  $ \pi_{i}=\lim_{t\to\infty}\limits \bold P[d_{(\cdot)}(t)=i ],$
$ i\in\overline{0,B+1} $
 be  the stationary  distribution  of the number of customers in the system
${\rm M^{\varkappa}|G^{\delta}|1|B}.$
Then
\begin{align*}
&    \pi_{0}=
    \left[1+\mu\bold E\eta\left(\frac{\lambda}{1-\lambda}
    \bold E Q_{\delta+B}+
    \sum_{i=0}^{B}\hat a_{i}Q_{B-i}\right)\right]^{-1},\\
&    \pi_{B+1}= 1-\pi_{0}(1+C_{B}(\lambda)),\qquad
    \pi_{i}=\pi_{0}\left(C_{i}(\lambda)-C_{i-1}(\lambda)\right),
    \quad i\in\overline{1,B},
\end{align*}
where
$
  \rho_{i}=\lim_{s\to 0}\limits s^{-1}\tilde\rho_{i}(s)=
  \int_{0}^{\infty}\bold P\left[\pi(t)=i\right]dt<\infty,
$
$ C_{0}(\lambda)=0,$
$$
    C_{u}(\lambda)=\frac{Q_{u}}{1-\lambda}-1+
    \lambda\mu\left( A_{0}^{u}- \bold E A_{0}^{\delta+u}\right),\quad
    A_{0}^{u}= \lim_{s\to 0}\limits \frac{1}{s}\,A_{0}^{u}(s)=
    \sum_{i=0}^{u} \rho_{i}\left[1-\frac{Q_{u-i}}{1-\lambda}\right].
$$
\end{corollary}

\begin{proof}
By
$
    \tilde p_{r,x}^{s}(u)=
    \bold P\left[d_{r,x}(\nu_{s})\le u; b_{r}(x)>\nu_{s} \right],
$
$ r,u\in\overline{1,B+1}$ denote the transient  probability   of the process
$ d_{r,x}(\nu_{s}) $  on the event $\{ b_{r}(x)>\nu_{s}\}. $
Taking  into account the  homogeneity of the  process  $ X_{t} $
with respect   to the first component, the  definition of   the process $ d_{r,x}(t)$
and the formulae (\ref{dmg1B34}) of Theorem  \ref{tdmg1B5},
we  derive
\begin{align*}
   \tilde p_{r,x}^{s}(B+1)=1-b_{r}^{s}(x),\qquad
   \tilde p_{r,x}^{s}(u)=A_{x}^{u-r}(s)-
   b_{r}^{s}(x) \bold EA_{0}^{\delta+u-1}(s),\quad
   u\in\overline{1,B},
\end{align*}
where  the function $ b_{r}^{s}(x) $  is determined   by (\ref{dmg1B39}).

In  accordance  with the definition of the process  $Y_{r,x}(t) $
we can write the following equations for the functions  $ q_{r,x}^{s}(u),$
$ q_{0,0}^{s}(u) $  for $ u\in\overline{1,B} $
\begin{align}                                    \label{dmg1B45}
&    q_{r,x}^{s}(u)= \tilde p_{r,x}^{s}(u)+
     b_{r}^{s}(x) q_{0,0}^{s}(u),\notag\\
&    q_{0,0}^{s}(u)= \frac{\mu}{s+\mu}
     \left[ \hat a_{B}\, q_{B+1,0}^{s}(u)+
    \sum_{i=1}^{B} a_{i}\, q_{i,0}^{s}(u) \right].
\end{align}
Substituting  the right-hand side  of the second  equation  into the first one, we  get
\begin{align*}
&    q_{r,x}^{s}(u)= \tilde p_{r,x}^{s}(u)+
     b_{r}^{s}(x) \frac{\mu}{s+\mu}\,\tilde q(s,u),
\end{align*}
where
$
    \tilde q(s,u)= \hat a_{B}\, q_{B+1,0}^{s}(u)+
    \sum_{i=1}^{B} a_{i}\, q_{i,0}^{s}(u).
$
Letting $ x=0 $ in the latter equality implies that
\begin{align*}
&  \hat a_{B}\, q_{B+1,0}^{s}(u)= \hat a_{B}\,\tilde p_{B+1,0}^{s}(u)+
    \hat a_{B}\, b_{B+1}^{s}(0) \frac{\mu}{s+\mu}\,\tilde q(s,u),\\
&  \sum_{i=1}^{B} a_{i}\, q_{i,0}^{s}(u)=
    \sum_{i=1}^{B} a_{i}\,\tilde p_{i,0}^{s}(u)+
    \sum_{i=1}^{B} a_{i}\, b_{i}^{s}(0) \frac{\mu}{s+\mu}\,\tilde q(s,u).
\end{align*}
Adding   these equalities,  we  obtain the function  $ \tilde q(s,u) $
\begin{align*}
   \tilde q(s,u)={\tilde p(s,u)}
    \left({1-\frac{\mu}{s+\mu}\,\tilde b(s)}\right)^{-1},
   \quad u\in\overline{1,B},
\end{align*}
where
$
\tilde b(s)=  \hat a_{B}\, b_{B+1}^{s}(0)+
    \sum_{i=1}^{B} a_{i}\, b_{i}^{s}(0),
$
\begin{align*}
    \tilde p(s,u)=
     \hat a_{B}\, \tilde p_{B+1,0}^{s}(u)+
    \sum_{i=1}^{B} a_{i}\,\tilde p_{i,0}^{s}(u) =
     \sum_{i=1}^{u} a_{i}\,A_{0}^{u-i}(s) -
     \tilde b(s) \bold E A_{0}^{\delta+u-1}(s).
\end{align*}
Substituting the expression for the function  $ \tilde q(s,u) $
into  (\ref{dmg1B45}), we find  that
\begin{align}                                              \label{dmg1B46}
&     q_{0,0}^{s}(u) = \bold E A_{0}^{\delta+u-1}(s)
    +\frac{C_{u}(s,\lambda)}{s+\mu-\mu\tilde b(s)},\notag\\
&    q_{r,x}^{s}(u)=  A_{x}^{u-r}(s) +b_{r}^{s}(x)
     \frac{C_{u}(s,\lambda)} {s+\mu-\mu\,\tilde b(s)},
\end{align}
where
$$
    C_{u}(s,\lambda)=\frac{sQ_{u}^{s}}{1-\lambda}-s +
    \lambda(s+\mu)\left( A_{0}^{u}(s)- \bold E A_{0}^{\delta+u}(s)\right).
$$
To derive the latter equalities, we used the following relation
$$
  \mu \sum_{i=1}^{u} a_{i}\,A_{0}^{u-i}(s) - (s+\mu) A_{0}^{u}(s)=
   \frac{s}{1-\lambda}\,Q_{u}^{s}-s,  \qquad  u\in\overline{1,\infty},
$$
which  follows from the definition   of the  function $ A_{x}^{u}(s). $
If $ u=B+1, $  then $\tilde p_{r,x}^{s}(B+1)=1-b_{r}^{s}(x),$
$\tilde p(s,B+1)=1-\tilde b(s)$
and  the  following  formulae are valid
\begin{align}                                              \label{dmg1B47}
     q_{0,0}^{s}(B+1) =1-
     \frac{s}{s+\mu-\mu\tilde b(s)},\qquad
     q_{r,x}^{s}(B+1)= 1- \frac{s b_{r}^{s}(x)}{s+\mu-\mu\tilde b(s)}.
\end{align}
Taking into account the definition of the process  $Y_{r,x}(t), $
we can write  the following   equations for  the  functions
$ q_{r,x}^{s}(0),$ $ q_{0,0}^{s}(0) $
\begin{align*}
&    q_{r,x}^{s}(0)= b_{r}^{s}(x) q_{0,0}^{s}(0),\\
&    q_{0,0}^{s}(0)=\frac{s}{s+\mu}+ \frac{\mu}{s+\mu}
     \left[ \hat a_{B} q_{B+1,0}^{s}(0)+
    \sum_{i=1}^{B} a_{i} q_{i,0}^{s}(0) \right].
\end{align*}
Solving this system  yields
\begin{align}                                               \label{dmg1B48}
     q_{0,0}^{s}(0) =
     \frac{s}{s+\mu-\mu\tilde b(s)},\qquad
     q_{r,x}^{s}(0)=  \frac{s\, b_{r}^{s}(x)}{s+\mu-\mu\,\tilde b(s)}.
\end{align}
Observe  that
$
    \lim_{s\to 0}\limits A_{x}^{u}(s)=
    \lim_{s\to 0}\limits\bold E A_{0}^{\delta+u}(s)=0,
$
$ \lim_{s\to 0}\limits b_{r}^{s}(x)=1.$
It  follows from   (\ref{dmg1B46})--(\ref{dmg1B48})  and  properties
of the Laplace transforms   that
\begin{align*}
&     \lim_{s\to 0}\limits q_{r,x}^{s}(u)=
     \lim_{s\to 0}\limits q_{0,0}^{s}(u)= q(u)=
     \lim_{t\to \infty}\limits \bold P[d_{(\cdot)}(t)\le u],\qquad
     u\in\overline{1,B+1}, \\
&     \lim_{s\to 0}\limits q_{r,x}^{s}(0)=
     \lim_{s\to 0}\limits q_{0,0}^{s}(0)= q(0)=
     \lim_{t\to \infty}\limits \bold P[d_{(\cdot)}(t)=0].
\end{align*}
The  formulae  (\ref{dmg1B48})  imply that
\begin{align*}
    q(0)= \lim_{s\to 0}\limits \frac{s}{s+\mu-\mu\tilde b(s)}=
    \left[1+\mu\bold E\eta\left(\frac{\lambda}{1-\lambda}
    \bold E Q_{\delta+B}+\sum_{i=0}^{B}\hat a_{i}Q_{B-i}\right)\right]^{-1}.
\end{align*}
In  view of (\ref{dmg1B46}), (\ref{dmg1B47}) we find
\begin{align*}
    q(B+1)= 1-q(0),\qquad q(u)=q{(0)}C_{u}(\lambda),
    \quad u\in\overline{1,B},
\end{align*}
where
$
  \rho_{i}=\lim_{s\to 0}\limits s^{-1}\tilde\rho_{i}(s)=
  \int_{0}^{\infty}\bold P\left[\pi(t)=i\right]dt<\infty,
$
$$
    C_{u}(\lambda)=\frac{Q_{u}}{1-\lambda}-1+
    \lambda\mu\left( A_{0}^{u}- \bold E A_{0}^{\delta+u}\right),\quad
    A_{0}^{u}= \lim_{s\to 0}\limits \frac{1}{s}\,A_{0}^{u}(s)=
    \sum_{i=0}^{u} \rho_{i}\left[1-\frac{Q_{u-i}}{1-\lambda}\right].
$$
\end{proof}

\subsection{  $ {\rm M^{\varkappa}|G|1|B} $ system }
Let us  stress  the following fact. If  we  set the  parameter
$\lambda=0$ in the  geometrical distribution
 $ \bold P[\delta=n]=(1-\lambda)\lambda^{n-1},$ $n\in\mathbb{N},$
$\lambda \in[0,1)$ of the random variable  $ \delta, $ then  $ \bold
P[\delta=1]=1.$ In  other  words it means  that  the process  $
\{D_{x}(t)\}_{t\ge 0} $ has  unit negative jumps   at the times
instants $\{\eta_{n}(x)\}_{n\in\mathbb{N}} $  and  $
\delta_{N_{x}(t)}=N_{x}(t).$ Then,
it follows from (\ref{dmg1B2}) that
\begin{align}                                   \label{dmg1B49}
   D_{x}(t)=
   \pi(t)-N_{x}(t)\in\mathbb{Z},\qquad t\ge0.
\end{align}
We will call this process  a difference of the compound Poisson process and a
simple renewal process. Setting the parameter $\lambda=0$   in  the
statements of Lemma \ref{ldmg1B1} leads to the following result.

\begin{lemma}  \label{ldmg1B5}

For $s>0$  the equation
$
\theta =\tilde f(s-k(\theta))
$
\\has a
unique  solution  $c(s)$ inside the circle $|\theta|<1.$ This
solution is positive,  $c(s)\in(0,1).$ If
$\bold E [\varkappa], \bold E [\eta]<\infty,$
$\rho= \mu\bold E[\varkappa]\,\bold E[\eta],$ then for  $\rho>1,$
$\lim_{s\to 0}\limits c(s)=c\in(0,1);$ and for  $\rho\le 1,$
$\lim_{s\to 0}\limits c(s)=1.$
\end{lemma}

The statements  of  Lemma's \ref{ldmg1B2}--\ref{ldmg1B4} and
Theorem's \ref{tdmg1B1}--\ref{tdmg1B6} can be reformulated in a similar way.
Letting  $\lambda=0$ in the defining formula
(\ref{dmg1B8}) for all $s,x\ge0$
we get
\begin{align}                                         \label{dmg1B50}
   Q_{k}^{s}(x) =\frac{1}{2\pi i}
   \oint_{|\theta|=\alpha}\,\frac1{\theta^{k+1}}\,
   \frac{\tilde f_x(s-k(\theta))}
   { \tilde f(s-k(\theta))-\theta}\,d\theta,
   \quad \alpha\in(0,c(s))
\end{align}
the  resolvent sequence of the process  $\{D_{x}(t)\}_{t\ge 0},$ which
is given by  (\ref{dmg1B49}). This resolvent sequence has been introduced in
\cite{Kad8}. Setting   $\lambda=0$ in
(\ref{dmg1B9}), (\ref{dmg1B10}), we obtain
\begin{align}                                         \label{dmg1B51}
   f^{k}_{x}(dl,m,s)
&  = e^{-s(l-x)}\,\frac{1-F(l)}{1-F(x)}\bold I\{l>x\}
   p_{k}^{m}(d(l-x))\notag\\
&  + \Phi^{s}_{0}(dl,m)\,Q_{k}^{s}(x)
   -e^{-sl}\,[1-F(l)]\sum_{i=0}^{k}
   Q_{i}^{s}(x)\,p_{k-i}^{m}(dl),\notag\\
  \bold E e^{-s\tau^{k}(x)}
&  = 1-\frac{s}{s-k(c(s))}\,Q_{k}^{s}(x)-A_{x}^{k}(s)
\end{align}
the Laplace transforms of  the upper one-boundary functionals  of
the process  $\{D_{x}(t)\}_{t\ge 0}$ (\ref{dmg1B49}), where
$
   \tilde\rho_{i}(s)=
   s\int_{0}^{\infty}e^{-st}\bold P[\pi(t)=i]\,dt,
$
$$
    A_{x}^{k}(s)=
    \sum_{i=0}^{k}\limits\tilde\rho_i(s)\,\left[1-Q_{k-i}^{s}(x)\right],
   \quad\Phi^{s}_{0}(dl,m)=e^{-sl}[1-F(l)]
   \sum_{k\in\mathbb{Z}^{+}}c(s)^{k}p_{k}^{m}(dl).
$$
We have introduced the  auxiliary functions  and the resolvent  sequence  of
the process (\ref{dmg1B49}), therefore we  can state the following result.

\begin{corollary}  \label{cdmg1B4}

Let  $\{D_{x}(t)\}_{t\ge 0}$  be the  difference of the compound Poisson process
and  the  renewal process (\ref{dmg1B49}),
$\{Q_{k}^{s}(x)\}_{k\in\mathbb{Z}^{+}},$  $x\ge0$
be the  resolvent sequence of the process given by
(\ref{dmg1B50}), $ Q_{k}^{s}\stackrel{\rm def}{=} Q_{k}^{s}(0).$  Then
\begin{itemize}
\item[{\rm(i)}]
the Laplace transforms  $ V^{x}_{r}(m,s),$ $ V^{k}_{x}(dl,m,s)$ of
the joint distribution of  $\{\chi,L,T\}$ satisfy  the  following
equalities for all  $x,s\ge0,$ $m\in\mathbb{N}$
\begin{align*}
&  V_{r}(x,i,s)=
   \frac{Q_{k}^{s}(x)}{Q_{B+1}^{s}}\,\delta_{i1},\;
   V^{k}(x,dl,i,s)= f^{k}(x,dl,i,s)-
   \frac{Q_{k}^{s}(x)}{Q_{B+1}^{s}}f^{B+1}(0,dl,i,s),
\end{align*}
where  $\delta_{ij}$ is the  Kronecker symbol and
$ f^{k}(x,dl,m,s) $ is given by (\ref{dmg1B51});
\item[{\rm(ii)}]
for the Laplace transforms of the first exit time $\chi$ from
the interval  by the process $\{D_{x}(t)\}_{t\ge 0}$ the formulae hold
\begin{align*}
&  \bold E\left[e^{-s\chi};\mathfrak{A}_{r}\right]=
   \frac{Q_{k}^{s}(x)}{Q_{B+1}^{s}},\quad
   \bold E\left[e^{-s\chi};\mathfrak{A}^{k}\right]=1- A_{x}^{k}(s)
  +\frac{Q_{k}^{s}(x)}{Q_{B+1}^{s}}(1- A_{0}^{B+1}(s));
\end{align*}
\item[{\rm(iii)}]
the exit probabilities from the interval by the process
$ \{D_{x}(t)\}_{t\ge 0} $ satisfy the equalities
$$
   \bold P[\mathfrak{A}_{r}]= \frac{Q_{k}(x)}{Q_{B+1}},\qquad
   \bold P[\mathfrak{A}^{k}]=1-\frac{Q_{k}(x)}{Q_{B+1}},
$$
where  the  resolvent sequence of the process
$\{Q_{k}(x)\}_{k\in\mathbb{Z}^{+}},$ $x\ge0,$
$ Q_{k}\stackrel{\rm def}{=} Q_{k}(0)$
is given by (\ref{dmg1B50}) for $ s=0. $
\end{itemize}
\end{corollary}

In order to prove the corollary, one has to put $\lambda=0$ in the statements
of Theorem \ref{dmg1B1}. The results obtained in  Corollary \ref{cdmg1B4}
can be  applied for studying  the queueing  systems
${\rm M^{\varkappa}|G|1|B}$ $ (\bold P[\delta=1]=1) $ with finite
waiting room.
To illustrate this, we   now  will  determine   the distribution
of the  busy period, the  number of the  customers,  time of the   first
loss of the  customer and  the number of lost customers at time of the first loss.
\begin{corollary}   \label{cdmg1B5}

Let   $ \bold P[\delta=1]=1, $
$ b_{r}^{s}(x)=\bold E\left[e^{-s b_{r}(x)}; b_{r}(x)<\infty\right] $
be the Laplace transform  of the   busy period  of type   $(r,x) $
 of  ${\rm M^{\varkappa}|G|1|B}$  system.
Then  the following    relation is valid
\begin{align*}
   b_{r}^{s}(x) =
   \frac{\tilde f_{x}(s)+
   (1-\tilde f(s))S^{s}_{B-r}(x)}
   {\tilde f(s)+ (1-\tilde f(s))
   S^{s}_{B}},\quad x\ge0,
    \quad r\in\overline{1,B+1},
\end{align*}
where
$
    S^{s}_{k}(x)=\sum_{i=0}^{k} Q_{i}^{s}(x),
$
$ S^{s}_{k}(x)=0,$ for $k<0.$
The random variable  $ b_{r}(x) $ is  proper and
\begin{align*}
   \bold E b_{r}(x)=
   \bold E\eta_{x}- \bold E\eta+
   \bold E\eta\left[ S_{B}-S_{B-r}(x) \right]<\infty,
\end{align*}
where $ S_{k}(x)=S_{k}^{0}(x).$
\end{corollary}
These formulae were derived in \cite{Kad8}. To prove the corollary,
it suffices to set $\lambda=0$ in the equalities of Corollary \ref{cdmg1B3}.

\begin{corollary}   \label{cdmg1B6}

Let $ \bold P[\delta=1]=1, $ $ l_{r}(x)$  be the time of the first loss of the  batch
of the  customers in the  system  ${\rm M^{\varkappa}|G|1|B},$ and
$ l_{r}^{s}(x)=\bold E\left[e^{-s l_{r}(x)}; l_{r}(x)<\infty\right] $
be the Laplace  transform of  $ l_{r}(x).$
Then  the following formula holds
\begin{align*}
    l_{r}^{s}(x)=1- A_{x}^{k}(s)-
    Q^{s}_{k}(x)
     \frac{s}{s+\mu}
    \left(1-\frac{\mu}{s+\mu}\,\tilde Q(s)/Q^{s}_{B+1}\right)^{-1},
\end{align*}
where  $ r\in\overline{0,B+1},$ $k=B+1-r,$
$
    \tilde Q(s)=
    \sum_{i=1}^{B+1}\bold P[\varkappa=i] Q_{B+1-i}^{s}.
$
The random variable $ l_{r}(x) $ is proper and has finite mathematical
expectation.
\end{corollary}

The function $ l_{r}^{s}(x) $  was found in \cite{Kad8} in a different form.

\begin{corollary}  \label{cdmg1B7}

The  distribution of the  number  of  the customers in the system
${\rm M^{\varkappa}|G|1|B}$  at  time  $ \nu_{s} $ is such that
for $r\in\overline{0,B+1},$ $ u\in\overline{1,B+1} $
\begin{align*}
&
     q_{r,x}^{s}(u)=  A_{x}^{u-r}(s) +b_{r}^{s}(x)
     \frac{s(Q_{u}^{s}-1)} {s+\mu-\mu\,\tilde b(s)},\quad
     q_{r,x}^{s}(B+1)= 1- \frac{s b_{r}^{s}(x)}{s+\mu-\mu\tilde b(s)},\\
&
     q_{r,x}^{s}(0,0)=
     \frac{s\, b_{r}^{s}(x)}{s+\mu-\mu\,\tilde b(s)}, \qquad
     b_{0}^{s}(x)\stackrel{\rm def}{=}1.
\end{align*}
\end{corollary}

\begin{corollary} \label{cdmg1B8}

Let  $ \pi_{i}=\lim_{t\to\infty}\limits \bold P[d_{(\cdot)}(t)=i ],$
$ i\in\overline{0,B+1} $
 be  the stationary distribution  of the number  of   the  customers in the system
${\rm M^{\varkappa}|G|1|B}.$
Then
\begin{align*}
&    \pi_{0}=
    \left(1+\mu\bold E\eta
    \sum_{i=0}^{B}\hat a_{i}Q_{B-i}\right)^{-1},\\
&    \pi_{B+1}= 1-\pi_{0}Q_{B},\qquad \pi_{i}=\pi_{0}(Q_{i}-Q_{i-1}),
    \quad i\in\overline{1,B},
\end{align*}
where  the  resolvent sequence of the process
$\{Q_{k}(x)\}_{k\in\mathbb{Z}^{+}},$ $ Q_{k}\stackrel{\rm def}{=} Q_{k}(0),$
$ Q_{0}=1 $ is given by (\ref{dmg1B50}) for $ s=0. $
\end{corollary}

To  prove Corollaries  \ref{cdmg1B6}--\ref{cdmg1B8}, it suffices to set
$\lambda=0$  in the  equalities of Corollaries   \ref{cdmg1B2}, \ref{cdmg1B3}
and Theorem  \ref{tdmg1B6}.   Note, that the formulae of Corollary
\ref{cdmg1B8} were  obtained in  \cite{Kad8}.

Suppose that the system starts functioning  from  the state  $(r,x),$
$ r\in\overline{0,B+1}.$ 
Denote  by  $ i_{r,x} $  the number of the
lost customers   at time  of the  first loss $l_{r}(x).$

\begin{corollary}   \label{cdmg1B9}

The  generating function
$
   L_{r,x}^{s}(z)= \bold E \left[e^{-sl_{r}(x)}z^{i_{r,x}} \right]
$
of the joint distribution  $ \{l_{r}(x),i_{r,x}\} $
is such that
\begin{align}                                       \label{qqq51}
    L_{r,x}^{s}(z)
&  =\frac{\mu}{s}\sum_{i=0}^{k}\limits E^{i}(z)
   \left[A^{k-i}_{x}(s)-A^{k-i-1}_{x}(s)\right]+\notag\\
&   + \mu\, \frac{Q_{k}^{s}(x)}{Q_{B+1}^{s}}\,
   \frac{\sum_{i=0}^{B+1}\limits E^{i}(z)
   \left[Q_{B+1-i}^{s}-Q_{B-i}^{s}\right]}
   {s+\mu-{\mu}\,\tilde Q(s)/Q^{s}_{B+1}},
   \qquad k=B+1-r,
\end{align}
\begin{align*}
  \bold E & \left[e^{-sl_{r}(x)};i_{r,x}=n\right]
    =\frac{\mu}{s}\left[a_{n}A^{k}_{x}(s)+
    \sum_{i=1}^{k}\limits (a_{n+i}-a_{n+i-1})A^{k-i}_{x}(s)\right]+ \\
&
   + \mu\,\frac{Q_{k}^{s}(x)}{Q_{B+1}^{s}}\,
   \frac{a_{n}Q_{B+1}^{s}+\sum_{i=1}^{B+1}\limits
   (a_{n+i}-a_{n+i-1})Q_{B+1-i}^{s}}
   {s+\mu-{\mu}\,\tilde Q(s)/Q^{s}_{B+1}},
   \qquad n\in\mathbb{N},
\end{align*}
where $ E^{i}(z)=\bold E\left[z^{\varkappa-i};\varkappa>i\right],$
$i\in\mathbb{Z}^{+}.$
\end{corollary}

\begin{proof}
In accordance with the total probability  law  we can write
the following system  of  equations  for  $ L_{r,x}^{s}(z) $

\begin{align}                              \label{qqq52}
&   L_{r,x}^{s}(z)=\tilde V^{k}(x,z,s)
    +V_{r-1}(x,s)L_{0,0}^{s}(z),
    \quad r\in\overline{1,B+1}\notag\\
&   L_{0,0}^{s}(z)=\frac{\mu}{s+\mu}\,E^{B+1}(z)+
    \frac{\mu}{s+\mu}\,\sum_{r=1}^{B+1}
    a_{r}L_{r,0}^{s}(z),
\end{align}
where
\begin{align}                                     \label{qqq53}
    \tilde V^{k}(x,z,s)=
    \bold E\left[e^{-s\chi_{r-1}^{B}(x)}
    z^{T};\mathfrak{A}^{k}\right]=
    \tilde f^{k}(x,z,s)-
    \frac{Q_{k}^{s}(x)}{Q_{B+1}^{s}}\,\tilde f^{B+1}(0,z,s).
\end{align}
The  equality  (\ref{dmg1B9})  implies for the  function
$
    \tilde f^{k}(x,z,s)= \bold E\left[e^{-s\tau^{k}(x)}
    z^{T^{k}(x)};\mathfrak{B}^{k}(x)\right]
$
that
\begin{align}                                     \label{qqq54}
    \tilde f^{k}(x,z,s)=
    \frac{\mu}{s}\sum_{i=0}^{k}\limits E^{i}(z)
    \left[A^{k-i}_{x}(s)-A^{k-i-1}_{x}(s)\right]+
    Q_{k}^{s}(x)F(c(s),z),
\end{align}
where $ F(c(s),z)=\frac{1-c(s)}{1-c(s)/z}\,\frac{k(z)-k(c(s)}{s-k(c(s)}.  $
Solving the  system (\ref{qqq52})  and taking into  account  (\ref{qqq53})
for all  $ r\in\overline{0,B+1},$ $x\ge 0,$ we get
\begin{align*}
    L_{r,x}^{s}(z)= \tilde f^{k}(x,z,s) +
   \mu\, \frac{Q_{k}^{s}(x)}{Q_{B+1}^{s}}\,
   \frac{E^{B+1}(z)+\sum_{i=1}^{B+1}\limits a_{i}\tilde f^{B+1-i}(0,z,s)
   -\tilde f^{B+1}(0,z,s)}
   {s+\mu-{\mu}\,\tilde Q(s)/Q^{s}_{B+1}}.
\end{align*}
The  first  formula of the  corollary   follows from the latter
equality  and  from  (\ref{qqq54}).  Comparing the coefficients   of
$z^{n},$ $n\in\mathbb{N}$ in both sides of (\ref{qqq51}), we derive the second
formula of the corollary.
\end{proof}


\end{document}